\documentclass[10pt]{amsart}
\usepackage{latexsym}
\usepackage{amsmath}
\usepackage{amssymb}
\usepackage{mathrsfs}
\usepackage{graphicx}
\usepackage{color}
\usepackage{pgfpages}
\usepackage{ifthen}
\usepackage{leftidx}
\IfFileExists{tensor.sty}{\usepackage{tensor}}{}
\usepackage[T1]{fontenc}
\usepackage[latin1]{inputenc}
\usepackage{mathtools}
\usepackage{comment}
\IfFileExists{dsfont.sty}{\usepackage{dsfont}}{}
\usepackage[nocompress]{cite}

\usepackage[shortlabels]{enumitem}
\usepackage{aliascnt}
\usepackage[bookmarks=true,pdfstartview=FitH, pdfborder={0 0 0}, colorlinks=true,citecolor=red, linkcolor=blue]{hyperref}
\IfFileExists{bbm.sty}{\usepackage{bbm}}{}
\usepackage{nicefrac}

\theoremstyle{plain}
\newtheorem{thm}{Theorem}[section]
\newaliascnt{cor}{thm}
\newaliascnt{prop}{thm}
\newaliascnt{lem}{thm}
\newtheorem{cor}[cor]{Corollary}

\newtheorem{prop}[prop]{Proposition}
\newtheorem{lem}[lem]{Lemma}
\aliascntresetthe{cor}
\aliascntresetthe{prop}
\aliascntresetthe{lem}

\theoremstyle{definition}
\newaliascnt{defn}{thm}
\newaliascnt{asu}{thm}
\newaliascnt{con}{thm}
\newtheorem{defn}[defn]{Definition}
\newtheorem{asu}[asu]{Assumption}

\aliascntresetthe{defn}
\aliascntresetthe{asu}
\aliascntresetthe{con}

\newcounter{stp}
\newcounter{stpi}
\newcounter{stpci}
\newcounter{stpiii}

\newtheorem{step}[stp]{Step}

\theoremstyle{thm}
\newaliascnt{rem}{thm}
\newaliascnt{exa}{thm}
\newaliascnt{masu}{thm}
\newaliascnt{nota}{thm}
\newaliascnt{sett}{thm}
\newtheorem{rem}[rem]{Remark}

\aliascntresetthe{rem}
\aliascntresetthe{exa}
\aliascntresetthe{masu}
\aliascntresetthe{nota}
\aliascntresetthe{sett}

\numberwithin{equation}{section}

\setlist[enumerate]{font = \normalfont}

\newcommand{\eps}{\varepsilon}
\renewcommand{\bar}[1]{\overline{#1}}

\newcommand{\R}{\mathbb{R}}
\newcommand{\C}{\mathbb{C}}
\newcommand{\N}{\mathbb{N}}

\newcommand{\E}{\mathbb{E}}

\newcommand{\rC}{\mathrm{C}}
\newcommand{\rL}{\mathrm{L}}
\newcommand{\rW}{\mathrm{W}}
\newcommand{\rH}{\mathrm{H}}
\newcommand{\rB}{\mathrm{B}}
\newcommand{\rD}{\mathrm{D}}

\newcommand{\rX}{\mathrm{X}}

\newcommand{\rF}{\mathrm{F}}

\newcommand{\rLq}{\rL^q}
\newcommand{\rLp}{\rL^p}

\newcommand{\Hinfty}{\mathcal{H}^\infty}

\renewcommand{\div}{\mathrm{div} \,}

\newcommand{\dk}[1]{\partial_{#1}}
\newcommand{\dt}{\dk{t}}

\renewcommand{\d}{\mathrm{d}}

\begin{document}

\title[Free-boundary value problem for stochastic compressible Navier--Stokes equations]{Noise-Driven Free Boundaries in the Compressible Navier-Stokes Equations}

\author{Gianmarco Del Sarto}
\address{Technische Universit\"{a}t Darmstadt\\
Fachbereich Mathematik\\
	Schlossgartenstr.\ 7\\
	64289 Darmstadt\\
	Germany}
\email{delsarto@mathematik.tu-darmstadt.de}

\author{Matthias Hieber}
\address{Technische Universit\"at Darmstadt\\
	Fachbereich Mathematik\\
	Schlossgartenstr.\ 7\\
	64289 Darmstadt\\
	Germany}
\email{hieber@mathematik.tu-darmstadt.de}

\author{Tarek Z\"{o}chling}
\address{Technische Universit\"at Darmstadt\\
	Fachbereich Mathematik\\
	Schlossgartenstr.\ 7\\
	64289 Darmstadt\\
	Germany}
\email{zoechling@mathematik.tu-darmstadt.de}

\maketitle
\begin{abstract}
A stochastic free-boundary problem for the three-dimensional barotropic compressible Navier--Stokes equations is studied. The main feature of the model is that the free boundary is transported by a Stratonovich stochastic flow, so that the noise enters the kinematic boundary condition and hence the evolution of the moving domain. An additional It\^o forcing in the momentum equation is also allowed.
The problem is transformed by a stochastic Lagrangian map generated by the velocity and the transport vector fields. In these coordinates the density is represented through the Jacobian of the flow, and the remaining system is solved by combining stochastic maximal regularity, deterministic \(\rL^p\)-\(\rL^q\) estimates, and a localized contraction argument. Local pathwise well-posedness is obtained up to an a.s. positive stopping time, with strictly positive density and pathwise uniqueness.
\end{abstract}

\section{Introduction}\label{sec:intro}
\noindent
Free-boundary problems for viscous fluids describe situations in which the region occupied by the fluid is itself unknown and evolves together with the fluid motion. Such problems arise naturally in the study of free surfaces, liquid drops, interfaces with vacuum, and moving-boundary models in continuum mechanics. From the mathematical point of view, the main difficulty is that the partial differential equations have to be solved on a time-dependent domain whose boundary is coupled to the solution itself.

\smallskip

\noindent
The deterministic theory of free-boundary problems for the Navier--Stokes equations has a long and rich history. For incompressible fluids, fundamental contributions go back to the work of Solonnikov and have since been developed in many directions using energy methods, Lagrangian coordinates, and maximal regularity techniques; see, for instance, \cite{Solonnikov,PadulaSolonnikov,ShibataFreeBoundary,ShibataMaxReg}. For a survey of the classical incompressible theory, we also refer to the handbook chapter of Denisova and Solonnikov \cite{SolonnikovDenisovaHandbook}. Related incompressible two-phase free-boundary problems are treated in the monograph of Pr\"uss and Simonett \cite{PruessSimonettBook} and \cite{KoehnePruessWilke}. For compressible viscous fluids, the problem is even more delicate because the motion of the density has to be coupled with the evolution of the free surface. In this direction, we mention in particular the works of Solonnikov--Tani as well as further developments by Enomoto and Shibata; see \cite{SolonnikovTani,EnomotoBelowShibata,EnomotoShibata}. For an overview of local and global solvability results for compressible free-boundary problems near equilibria, we refer to the handbook \cite{DenisovaSolonnikovHandbook}. In these works, a central role is played by a transformation to a fixed reference domain and by estimates for the corresponding linearized free-boundary problem.

\smallskip

\noindent
In contrast, stochastic effects in free-boundary and moving-boundary problems are much less understood. Random perturbations are relevant in many physical situations, for instance when unresolved microscopic effects, irregular environmental forcing, or random fluctuations influence the motion of the fluid or of the boundary. Several recent works have started to address stochastic moving-boundary problems in the related context of fluid--structure interaction. We mention, for example, the work of Kuan and \v{C}ani\'c on stochastic fluid--structure interaction models with noise acting on the structure, as well as subsequent developments and recent work on fluid--structure interaction with transport noise and randomly moving boundaries; see \cite{KuanCanicJDE,KuanCanicJMFM,KuanOhCanic,BreitFSI}. {Related transport-type stochastic perturbations of elastic interfaces arise in the variational derivation of Koiter shell models developed in \cite{MensahNgougoue2026}.} These results show that stochastic perturbations in coupled fluid-boundary systems lead to genuinely new analytical difficulties.

\smallskip

\noindent
The aim of the present paper is to study a stochastic free-boundary problem for the barotropic compressible Navier--Stokes equations. The fluid occupies an unknown moving domain \(\Omega_t\subset\mathbb R^3\), and the unknowns are the density, the velocity, and the free boundary \(\Gamma_t=\partial\Omega_t\). The main feature of the model is that the stochastic perturbation is tied to the motion of the domain itself. More precisely, we include Stratonovich transport noise, which represents a random advective component in the motion of fluid particles. Since the free boundary is a material surface, this random transport also has to enter the kinematic law for the interface.

\smallskip

\noindent
In the deterministic free-boundary problem, particles that start on the interface remain on the interface, and the boundary is transported by the fluid velocity. In the present stochastic setting, the particle motion contains both the velocity field and the random transport field. Consequently, the kinematic boundary condition is given by
\[
\d x_t
=
u(t,x_t)\,\d t
+
\sum_{k=1}^K Q_k(x_t)\circ\d W_t^k \ \text{ for } \  x_t\in\Gamma_t,
\]
where \(Q_k\) are prescribed transport vector fields and \(W^k\) are independent Brownian motions. Thus, the free boundary is transported by the full stochastic flow, not only by the velocity field \(u\). This distinguishes the problem from stochastic fluid models in which the noise appears only as an external force on a fixed or deterministically moving domain. We also allow for an additional It\^o forcing in the momentum equation, but the central point of the model is the coupling between transport noise and the free-boundary motion.

\smallskip

\noindent
Our main result proves local-in-time well-posedness of this stochastic free-boundary problem in a strong pathwise sense. More precisely, for initial data with strictly positive density and satisfying the natural compatibility condition at the initial boundary, we construct an a.s. positive stopping time \(\tau\) and a unique local strong solution on \([0,\tau]\). The density remains strictly positive, and the free boundary satisfies the stochastic kinematic condition above. In particular, the result gives a stochastic counterpart of the deterministic local well-posedness theory for compressible viscous fluids with free boundary.

\smallskip

\noindent
The proof is based on a stochastic Lagrangian transformation. More precisely, for a given transformed velocity \(\bar u\) and $y \in \Omega_0$, we consider the flow
\[
\left\{
\begin{aligned}
\d \rX(t,y)
&=
\bar u(t,y)\,\d t
+
\sum_{k=1}^K Q_k(\rX(t,y))\circ \d W_t^k,
\\
\rX(0,y)&=y.
\end{aligned}
\right.
\]
Thus the transformation consists of two parts: the velocity part, which depends on the unknown fluid motion, and the stochastic part, which is generated by the prescribed transport vector fields. The moving domain is then represented as
\[
\Omega_t=\rX(t,\Omega_0),
\]
and the original problem is pulled back to the fixed reference domain \(\Omega_0\). A key point in the analysis is to construct a stopping time on which this stochastic Lagrangian map remains close to the identity and its inverse, Jacobian, and transformed coefficients satisfy suitable pathwise bounds.

\smallskip

\noindent
This approach is based on a stochastic flow generated by the transport vector fields and the velocity. A related stochastic flow with a prescribed deterministic drift was considered by Flandoli, Gubinelli, and Priola \cite{FlandoliGubinelliPriola} in the study of transport equations. Related stochastic models for compressible Navier--Stokes equations have been studied in several works. This includes the weak martingale solution theory of Breit and Hofmanov\'a \cite{BreitHofmanova}, as well as the relative-energy framework, weak--strong uniqueness, local strong solutions, and the general theory developed by Breit, Feireisl, and Hofmanov\'a \cite{BreitFeireislHofmanovaRelativeEnergy,BreitFeireislHofmanovaStrong,BreitFeireislHofmanovaBook}. Compressible fluid models with transport noise were treated in \cite{BreitFeireislHofmanovaZatorska,BreitFeireislHofmanovaMucha}. These settings do not involve a free-boundary value problem; in contrast, here the stochastic flow determines the evolution of the physical domain itself.
\smallskip

\noindent
Under the stochastic Lagrangian transformation, the transport noise is absorbed into the coefficients of the transformed system. Moreover, mass conservation gives an explicit expression for the transformed density $\bar \varrho$ by
\[
\bar\varrho(t,y)=\frac{\varrho_0(y)}{J(t,y)}
\ \text{ where } \
J(t,y)=\det\nabla\rX(t,y).
\]
Thus the fixed-point argument can be formulated in terms of the transformed velocity alone. The additive It\^o forcing is separated by writing the transformed velocity as the sum of a stochastic convolution and a pathwise deterministic remainder. The proof then combines pathwise estimates for the stochastic Lagrangian flow, stochastic maximal regularity for the additive part, deterministic maximal \(\rL^p\)-\(\rL^q\) regularity for the Lam\'e system, and a localized contraction argument.

\smallskip

\noindent
The paper is organized as follows. In \autoref{sec: prelim} we introduce the stochastic free-boundary model, describe the kinematic boundary condition, and fix the functional-analytic framework. In \autoref{sec: main} we state the main local well-posedness result. In \autoref{sec: Transform} we construct the stochastic Lagrangian transformation, derive the transformed problem on the fixed reference domain, and collect the required pathwise estimates for the transformed coefficients. In the following section we analyze the corresponding linearized Lam\'e system and separate the additive stochastic forcing by means of a stochastic convolution. In \autoref{sec: estimates} we prove the nonlinear estimates for the transformed velocity equation and the boundary condition. Finally, in \autoref{sec: local} we combine these ingredients in a localized contraction argument and reconstruct the Eulerian free-boundary solution. The appendix contains auxiliary estimates for the stochastic Lagrangian flow.

\section{Preliminaries}\label{sec: prelim}
\noindent
Let $T>0$ and consider a three-dimensional barotropic compressible viscous fluid occupying a moving domain
\[
\Omega_t \subset \mathbb{R}^3 \ \text{ with }\  t \in [0,T],
\]
where $\Omega_0\subset \mathbb R^3$ is a bounded $\rC^2$ domain,
The unknowns are the density $\varrho = \varrho(t,x) \ge 0$,
the velocity $u = u(t,x) \in \mathbb{R}^3,$ and the free boundary $\Gamma_t := \partial \Omega_t.$
We assume the pressure law
\[
p(\varrho) = a \varrho^\gamma \ \text{ where }\  a>0 \ \text{ and }\  \gamma>1,
\]
and the viscous stress tensor
\[
\mathbb{S}(\nabla u)
=
\mu \bigl(\nabla u + \nabla u^{\top}\bigr)
+
\lambda \, \operatorname{div} u \, \mathbb{I},
\]
where $\mu >0$ and $2\mu + 3\lambda>0$.
Let \(Q_1,\dots,Q_K : \mathbb{R}^3 \to \mathbb{R}^3\) be given smooth, divergence-free transport vector fields. {This assumption models volume-preserving random transport, while compression and expansion remain possible through $\div u$.}
The transport noise in the continuity equation is of the form
\[
\sum_{k=1}^K \operatorname{div}(\varrho Q_k)\circ \d W_t^k,
\]
whereas in the momentum equation it is given by
\[
\sum_{k=1}^K \operatorname{div}(\varrho u\otimes Q_k)\circ \d W_t^k,
\]
see also \cite{BreitFeireislHofmanovaMucha}.
We work on a complete filtered probability space
\(
(\Omega,\mathcal F,(\mathcal F_t)_{t\in[0,T]},\mathbb{P}).
\)
Let \((W^i)_{i=1}^K\) be independent real-valued \((\mathcal F_t)\)-adapted Brownian motions, and let \(B\) be a cylindrical \((\mathcal F_t)_t\)-Wiener process on a separable Hilbert space \(\mathcal U\) independent of $W^k$ for $k=1,...,K$. {We use finitely many transport fields to apply smooth stochastic-flow estimates directly. Infinite sums would require summability of their spatial derivatives and corresponding flow estimates.}
If \(\mathrm{Y}\) is a Banach space, an \(\mathrm{Y}\)-valued stochastic process \(\Phi\) is adapted if \(\Phi(t)\) is \(\mathcal F_t\)-measurable for every \(t\in[0,T]\). It is progressively measurable if the map
\[
(s,\omega)\mapsto \Phi_s(\omega)
\]
is measurable on
\(
[0,t]\times\Omega,
\,
\mathcal B([0,t])\otimes\mathcal F_t,
\)
for every \(t\ge0\), where \(\mathcal B([0,t])\) denotes the Borel \(\sigma\)-algebra on \([0,t]\).
Further, we denote by
\[
\rL^0_{\mathscr P}(\Omega;\rL^p(0,T;\mathrm{Y}))
\]
the space of all progressively measurable processes
\(f:(0,T)\times\Omega\to \mathrm{Y}\) such that
\(
f(\cdot,\omega)\in \rL^p(0,T;\mathrm{Y})
\)
for a.e. \(\omega\in\Omega\). Moreover, given a sub-\(\sigma\)-algebra \(\mathcal G\subset\mathcal F\), we denote by
\(
\rL^0_{\mathcal G}(\Omega;\mathrm{Y})
\)
the space of \(\mathcal G\)-measurable \(\mathrm{Y}\)-valued random variables.
The transport noise is understood in the \emph{Stratonovich} sense, whereas the stochastic forcing {is understood in the \emph{It\^o} sense. We prescribe \(f_{\mathrm{sto}}(t,y)\) on \(\Omega_0\). Its Eulerian expression is \(f_{\mathrm{sto}}(t,\rF(t,x))\), where \(\rF(t,\cdot)=\rX(t,\cdot)^{-1}\) is} the {inverse material flow.} In particular, for sufficiently smooth vector fields \(Q_k\) and a semimartingale \(Z\), one has
\[
\sum_{k=1}^K Q_k(Z_t)\circ \d W_t^k
=
\sum_{k=1}^K Q_k(Z_t)\,\d W_t^k
+
\frac12\sum_{k=1}^K DQ_k(Z_t)\,Q_k(Z_t)\,\d t.
\]
For background on Stratonovich and It\^o integration we refer to \cite{MR1121940, MR1725357}.

\smallskip

\noindent
The compressible Navier-Stokes free-boundary problem with transport noise and stochastic momentum forcing is then given by the system
\begin{equation}
\left\{
\begin{aligned}
\d \varrho + \operatorname{div}(\varrho u)\,\d t
& = - \sum_k \operatorname{div}(\varrho Q_k)\circ \d W_t^k ,
&& \text{in } \Omega_t \times (0,T),
\\
\d (\varrho u)
+\Bigl[
\operatorname{div}(\varrho u \otimes u)
+\nabla p
-\operatorname{div}\mathbb{S}(\nabla u)
\Bigr]\d t
&=
-\sum_{k} \operatorname{div}(\varrho u \otimes Q_k)\circ \d W_t^k
\\&\quad +
\varrho {  ( f_{\rm sto} (t) \circ F(t))}\d B_t,
&& \text{in } \Omega_t \times (0,T),
\\
\d x_t
&=
u(t,x_t)\,\d t
+\sum_k Q_k(x_t)\circ \d W_t^k,
&& \text{on } \Gamma_t \times (0,T),
\\
\bigl(\mathbb{S}(\nabla u)-p(\varrho)\mathbb{I}_3\bigr)n_t
&=
-p_{\mathrm{ext}}\,n_t,
&& \text{on } \Gamma_t \times (0,T),
\\
\varrho(0,x)=\varrho_0(x),
\quad
u(0,x)&=u_0(x), && \text{in } \Omega_0.
\end{aligned}
\right.
\label{eq:stochastic-free-boundary-CNSE-RHS}
\end{equation}
{Here \(p_{\mathrm{ext}}\ge0\) is the prescribed constant exterior pressure.}
The boundary condition
\[
\d x_t
=
u(t,x_t)\,\d t
+
\sum_{k} Q_k(x_t)\circ \d W_t^k,
\qquad x_t\in \Gamma_t,
\]
is the \emph{stochastic kinematic boundary condition}. Its physical meaning is that the free boundary is a \emph{material surface}. Fluid particles that start on the interface remain on the interface for later times. In the deterministic free-boundary theory, this principle is expressed by the condition
\[
\dt x_t=u(t,x_t),
\qquad x_t\in \Gamma_t,
\]
which says that the free surface is transported by the fluid velocity. In particular, there is no flux of fluid through the free boundary, and the motion of the interface is completely determined by the motion of the fluid particles located on it.
A convenient geometric way to describe the moving boundary is by means of a \emph{level-set function}. More precisely, let
\[
\Phi=\Phi(t,x)
\]
be such that
\[
\Gamma_t=\{x\in \mathbb R^3:\Phi(t,x)=0\} \ \text{ and }\ 
\Omega_t=\{x\in \mathbb R^3:\Phi(t,x)<0\}.
\]
Then the outward unit normal is given by
\[
n_t=\frac{\nabla \Phi(t,x)}{|\nabla \Phi(t,x)|}
\ \text{ on }\Gamma_t.
\]
In the deterministic setting, the fact that the boundary is transported by the flow is equivalent to the \emph{kinematic level-set equation}
\[
\partial_t \Phi + u\cdot \nabla \Phi=0
\ \text{ on }\Gamma_t.
\]
Taking the normal component, one obtains the equivalent geometric form
\[
V_{\Gamma_t}=u\cdot n_t,
\]
where \(V_{\Gamma_t}\) denotes the normal velocity of the free boundary.
In the present stochastic model, the bulk equations contain \emph{transport noise}
\[
\sum_k Q_k(x)\circ \d W_t^k,
\]
which is interpreted as an additional random advective component of the fluid parcel motion. Consequently, the correct stochastic counterpart of the kinematic boundary condition is that the interface is transported by the \emph{full stochastic velocity field},
\[
u\,\d t+\sum_k Q_k\circ \d W_t^k.
\]
Therefore, a boundary particle trajectory \(x_t\in\Gamma_t\) must satisfy
\[
\d x_t
=
u(t,x_t)\,\d t
+
\sum_k Q_k(x_t)\circ \d W_t^k.
\]
Equivalently, in level-set form, the boundary is characterized by the \emph{stochastic kinematic condition}
\[
\d \Phi + u\cdot \nabla \Phi\,\d t
+
\sum_k Q_k\cdot \nabla \Phi \circ \d W_t^k
=0
\ \text{ on }\Gamma_t.
\]
Thus, the stochastic term appears in the boundary condition for the same reason as in the bulk equations: it represents random transport of fluid particles. Since the free boundary is a material surface, its motion must be consistent with that same stochastic transport law.

\smallskip

\noindent
Let us also comment on the energy structure of the model. Denote by \(P\) the pressure potential, defined by
\[
P'(\varrho)\varrho-P(\varrho)=p(\varrho).
\]
Formally, the natural total energy is
\[
\int_{\Omega_t}
\left(
\frac12\varrho |u|^2+P(\varrho)
\right)\,\d x
+
p_{\mathrm{ext}}|\Omega_t|.
\]
The last term accounts for the work of the constant exterior pressure on the moving boundary. Since the boundary is transported by the same stochastic flow as the fluid particles, the transport-noise contributions have the form of stochastic fluxes through the moving domain. In particular, the assumption
\(
\operatorname{div} Q_k=0,
\)
is important for the energy balance. It prevents the appearance of additional bulk production terms coming from the stochastic transport. This point is also emphasized in the work of Breit, Feireisl, Hofmanov\'a, and Zatorska on the compressible Navier--Stokes system with transport noise \cite{BreitFeireislHofmanovaZatorska}. In the absence of the It\^o forcing \(f_{\mathrm{sto}}\,\d B_t\), the formal energy balance therefore has the same dissipative structure as in the deterministic free-boundary problem, namely
\[
\d
\left[
\int_{\Omega_t}
\left(
\frac12\varrho |u|^2+P(\varrho)
\right)\,\d x
+
p_{\mathrm{ext}}|\Omega_t|
\right]
+
\int_{\Omega_t}\mathbb S(\nabla u):\nabla u\,\d x\,\d t
=0,
\]
up to the usual interpretation of the preceding identity along the stochastic flow. The additional It\^o forcing in the momentum equation produces the corresponding martingale and It\^o correction terms.

\smallskip

\noindent
We remark that in the case of vanishing noise coefficients the system \eqref{eq:stochastic-free-boundary-CNSE-RHS} coincides with the deterministic free boundary value problem for the compressible Navier-Stokes equations.
\begin{rem}
    {\rm
If $Q_k \equiv 0\text{ for all } k=1,...,K,
\text{ and }
f_{\mathrm{sto}} \equiv 0$
then the stochastic terms disappear and the system reduces to the standard free-boundary value problem for the three-dimensional barotropic compressible Navier--Stokes equations:
\begin{equation}
\left\{
\begin{aligned}
\partial_t \varrho + \operatorname{div}(\varrho u)
&= 0,
&& \text{in } \Omega_t \times (0,T),
\\
\partial_t(\varrho u)
+\operatorname{div}(\varrho u \otimes u)
+\nabla p(\varrho)
-\operatorname{div}\mathbb S(\nabla u)
&= 0,
&& \text{in } \Omega_t \times (0,T),
\\
V_{\Gamma_t}
&= u \cdot n_t,
&& \text{on } \Gamma_t \times (0,T),
\\
\bigl(\mathbb S(\nabla u)-p(\varrho)I_3\bigr)n_t
&= -p_{\mathrm{ext}}\,n_t,
&& \text{on } \Gamma_t \times (0,T),
\\
\varrho(0,x)=\varrho_0(x),
\quad
u(0,x)&=u_0(x),
&& \text{in } \Omega_0.
\end{aligned}
\right.
\label{eq:deterministic-free-boundary-CNSE}
\end{equation}
In this case the free boundary is transported only by the fluid velocity, and the stochastic kinematic condition reduces to the usual deterministic one
\[
V_{\Gamma_t}=u\cdot n_t.
\]
Hence the present formulation is a genuine stochastic extension of the classical compressible Navier--Stokes free-boundary problem.
}
\end{rem}

\subsection{Functional setting}
Some words about the functional-analytic setting are in order. Let \(s\in\R\) and \(p,q\in(1,\infty)\). For an open set \(\mathcal O\subset\R^3\), we denote by \(\rW^{s,q}(\mathcal O)\) the Sobolev--Slobodeckij spaces, by \(\rH^{s,q}(\mathcal O)\) the Bessel potential spaces, by \(\dot{\rH}^{s,q}(\mathcal O)\) the corresponding homogeneous Bessel potential spaces, by \(\rB^s_{qp}(\mathcal O)\) the Besov spaces, and by \(\rF^s_{pq}(\mathcal O)\) the Triebel--Lizorkin spaces. As usual, we set
\[
\rH^{0,q}(\mathcal O):=\rL^q(\mathcal O),
\]
and note that, for \(m\in\N_0\),
\[
\rH^{m,q}(\mathcal O)=\rW^{m,q}(\mathcal O)
\]
with equivalent norms. For details concerning these spaces, we refer to the monographs \cite{Tri:78,Adams,Ama:19}.
In the sequel, the function spaces on the fixed reference domain \(\Omega_0\) are always understood in this standard sense. {For a Hilbert space \(\mathcal U\) and a Banach space \(E\), let \(\gamma_{\mathrm{rad}}(\mathcal U,E)\) denote the space of \(\gamma\)-radonifying operators from \(\mathcal U\) to \(E\). If \(E\) is a Hilbert space, this space coincides isometrically with the Hilbert--Schmidt class.  For further details, see \cite[Chapter 9]{Hytnen2017}. We are now in a position to state the notion of solutions used in this article.}
\begin{defn}[Local strong solutions of \eqref{eq:stochastic-free-boundary-CNSE-RHS}]
\label{def: local strong solution}\mbox{}\\
{Fix \(p\in(2,\infty)\) and \(q \in ( 3, \infty)\) with
\(\frac{2}{p} + \frac{3}{q}<1\), and work on the filtered probability space
\((\Omega,\mathcal F,(\mathcal F_t)_{t\in[0,T]},\mathbb P)\)
with Wiener processes \((W^k)_{k=1}^K\) and \(B\) introduced above.} Let $\tau$ be an a.s. positive $(\mathcal{F}_t)_t$-stopping time. A triple $(\varrho,u, (\Omega_t)_{t\in[0,\tau]}) $
is a \emph{local strong solution} of \eqref{eq:stochastic-free-boundary-CNSE-RHS} on \([0,\tau]\) if the following hold.
\begin{enumerate}
\item[(i)] There exists a progressively measurable map $
\Lambda:[0,\tau]\times\Omega\times\Omega_0\to\mathbb R^3$
such that, for a.e. $\omega$ and every $t\in[0,\tau(\omega)]$, the map $
\Lambda(t,\omega,\cdot):\Omega_0\to\Omega_t(\omega) $
is a $\rC^1$-diffeomorphism into the bounded domain $\Omega_t(\omega)$, with
$$
\Omega_t(\omega)=\Lambda(t,\omega,\Omega_0) \ \text{ and }\ \Gamma_t(\omega)=\Lambda(t,\omega,\Gamma_0).
$$
{Here \(\Lambda\) is the material flow, that is \(\Lambda(0,y)=y\) and \(\d\Lambda(t,y) = u (t,\Lambda(t,y)) \d t + \sum_kQ_k(\Lambda(t,y))\circ\d W_t^k\).}
\item[(ii)]
For $y\in \Omega_0$, the pullbacks 
$$
\bar\varrho(t,y):=\varrho\bigl(t,\Lambda(t,y)\bigr) \ \text{ and }\ \bar u(t,y):=u\bigl(t,\Lambda(t,y)\bigr),
$$
are progressively measurable with values in
\(\rH^{1,q}(\Omega_0)\) and \(\rL^q(\Omega_0;\mathbb R^3)\), respectively, and satisfy
\begin{equation*}
    \begin{aligned}
        \bar\varrho &\in \rL_\mathscr{P}^0 (\Omega;\rH^{1,p}(0,\tau;\rH^{1,q}(\Omega_0))), \\ \quad 
\bar u &\in \rL_\mathscr{P}^0 \big (\Omega;\bigcap_{\theta \in [0, \nicefrac{1}{2})}  \rH^{\theta,p}(0,\tau ;\rH^{2- 2\theta,q}(\Omega_0 ;  \R^3)) \cap \rC([0,\tau];\rB^{2-\nicefrac{2}{p}}_{qp}(\Omega_0;\R^3))
\big ).
    \end{aligned}
\end{equation*} 

\item[(iii)]
The system \eqref{eq:stochastic-free-boundary-CNSE-RHS} is satisfied almost surely on \([0,\tau]\) in the strong sense.
\end{enumerate}
\end{defn}

\section{Main Results}\label{sec: main}
\noindent
In this section, we state the main result of this work, namely the local strong well-posedness of the stochastic free-boundary problem \eqref{eq:stochastic-free-boundary-CNSE-RHS} under suitable assumptions on the initial data and the stochastic forcing.
\begin{thm}[Local strong well-posedness of \eqref{eq:stochastic-free-boundary-CNSE-RHS}] \mbox{} \\ 
\label{thm:local-eulerian}
\noindent
Let \(p\in(2,\infty)\), \(q\in(3,\infty)\) satisfy
\(
\frac{2}{p}+\frac{3}{q}<1,
\)
let \(\Omega_0\subset\R^3\) be a bounded domain of class \(\rC^2\), and assume that
\(
Q_k\in \rC_b^\infty(\R^3;\R^3)  \text{ with }  \div Q_k =0  \text{ for all }  k=1,\dots,K.
\)
Assume moreover $$
\varrho_0 \in \rL^0_{\mathcal{F}_0}(\Omega;\rH^{1,q}(\Omega_0)) \ \text{ and }\ u_0 \in \rL^0_{\mathcal{F}_0}(\Omega;\rB^{2-\nicefrac{2}{p}}_{qp}(\Omega_0;\R^3)),$$
and that there exist constants \(0<\varrho_\ast\le \varrho^\ast<\infty\) such that
\[
\varrho_\ast \le \varrho_0(\omega,x)\le \varrho^\ast
\
\text{ for a.e. } (\omega,x)\in\Omega\times\Omega_0.
\]
Additionally, assume that the compatibility condition
\[
\bigl(\mathbb S(\nabla u_0(\omega,\cdot))
-
p(\varrho_0(\omega,\cdot)) I_3\bigr)N
=
-p_{\mathrm{ext}}\,N
\
\text{ for a.e. }(\omega,x) \in \Omega \times \Gamma_0
\]
is satisfied. Moreover, let {$0 < \eps < 1/(2q)$} and
\[
f_{\mathrm{sto}}
\in
\rL^0_{\mathscr{P}}\bigl(\Omega;\rL^p\bigl(0,T;{\gamma_{\mathrm{rad}}}(\mathcal{U},\rH^{1+2\eps,q}(\Omega_0 ; \R^3) )\bigr)\bigr).
\]
Then there exists an a.s. positive \((\mathcal F_t)_t\)-stopping time \(\tau\)
and a local strong solution
\[
(\varrho,u,\{\Omega_t\}_{t\in[0,\tau]})
\]
of \eqref{eq:stochastic-free-boundary-CNSE-RHS} on \([0,\tau]\) in the sense of
\autoref{def: local strong solution}.
Moreover, the density remains strictly positive on the stochastic interval
\([0,\tau]\), in the sense that
\[
\varrho(t,\omega,x)\ge \frac{\varrho_\ast}{2}
\quad
\text{for all } (t,x)\in[0,\tau(\omega)]\times\Omega_t(\omega)
\ \text{ for a.e. }\omega\in\Omega 
\]
and the moving boundary satisfies the stochastic kinematic condition
\[
\d x_t
=
u(t,x_t)\,\d t
+
\sum_{k=1}^K Q_k(x_t)\circ \d W_t^k
\
\text{ on }\Gamma_t\times(0,\tau).
\]
Finally, the solution is pathwise unique in the above class.
\end{thm}

\section{Stochastic Lagrangian transformation}\label{sec: Transform}
\noindent
Let \(T>0\). In this section we transform the free-boundary value problem \eqref{eq:stochastic-free-boundary-CNSE-RHS} to the fixed reference domain \(\Omega_0\). Since the local well-posedness argument is carried out pathwise, we fix \(\omega\in\Omega\) throughout and assume that
\[
f_{\mathrm{sto}}(\omega)\in \rL^p(0,T;{\gamma_{\mathrm{rad}}}(\mathcal{U},\rH^{1+ 2\eps,q}(\Omega_0))),
\]
and denote by \(U(\omega)\) the corresponding stochastic convolution, that is, the unique strong adapted solution of \eqref{eq:U-system}. Let
\begin{equation}
    \label{eq: assumption flow}
    v(\omega)\in {\mathbb E_{1,T} }  :=\rH^{1,p}(0,T;\rLq(\Omega_0))
\cap
\rLp(0,T;\rH^{2,q}(\Omega_0)) \ \text{ and set }\ 
\bar u(\omega):=v(\omega)+U(\omega).
\end{equation}
For fixed \(\omega\), we define the stochastic flow map \(\rX(\omega)\) as the solution of
\begin{equation}
\label{eq:flowX}
\left\{
\begin{aligned}
\d \rX(t,y)
&=
\bar u(\omega,t,y)\,\d t
+
\sum_k Q_k\bigl(\rX(t,y)\bigr)\circ \d W_t^k(\omega),
\quad t \in (0,T),
\\
\rX(0,y)
&=
y,
\quad y\in \Omega_0.
\end{aligned}
\right.
\end{equation}
The flow equation may be written in the integral form
\begin{equation}
\label{eq:stochastic-flow-lemma}
\rX(t,y)
=
y+\int_0^t \bar u(\omega,s,y)\,\d s
+
\sum_k\int_0^t Q_k\bigl(\rX(s,y)\bigr)\circ \d W_s^k(\omega),
\quad t\in[0,T].
\end{equation}
With the aid of the stochastic flow map \(\rX(\omega)\), the evolving fluid region is described by
\[
\Omega_t(\omega)=\rX(t,\Omega_0)
\ \text{ and }\
\Gamma_t(\omega)=\rX(t,\Gamma_0).
\]
In this way, the original free-boundary problem is rewritten on the fixed reference domain \(\Omega_0\).
We emphasize that this change of variables is not only a device for fixing the domain. Because \(\rX\) is generated by the full stochastic transport field, the transport terms are absorbed into the time derivative in Lagrangian coordinates. Thus, both the hyperbolic transport effects and the stochastic transport effects are hidden in the transformation itself. They reappear through the transformed coefficients, which depend on \(\rX\), its gradient, its inverse, and its Jacobian. For this reason, a careful analysis of these quantities is a central part of the argument.

\smallskip

\noindent
On the local interval where \(\rX\) is a diffeomorphism, define  the inverse flow map \(\rF(\omega)\) by
\[
\rF(\omega,t,\cdot):=\rX(\omega,t,\cdot)^{-1},
\]
and set
\[
\mathrm Z(\omega,t,\cdot):=\bigl(\nabla \rX(\omega,t,\cdot)\bigr)^{-1}
=
\nabla_x \rF\bigl(\omega,t,\rX(\omega,t,\cdot)\bigr)
\ \text{ and } \
J(\omega,t,\cdot):=\det \nabla \rX(\omega,t,\cdot).
\]
The following lemma establishes the pathwise well-posedness of the stochastic Lagrangian transformation introduced above. 
The lemma reads as follows.
\begin{prop}[Pathwise estimates for the stochastic Lagrangian transform] \mbox{} \\
\label{prop:pathwise-stochastic-lagrangian}
Let \(0 < T\le 1 \) , \(p\in(2,\infty)\), \(q\in(3,\infty)\), { $0 < \eps < 1/(2q)$}, $f_{\mathrm{sto}}(\omega) \in \rL^p(0,T;{ \gamma_{\mathrm{rad}}}(\mathcal{U},\rH^{1+ 2\eps,q}(\Omega_0)))$ and fix
\(\omega\in\Omega\). Assume that
\[
v(\omega)\in \begingroup \mathbb E_{1,T}\endgroup  \ \text{ and } \
\bar u(\omega):=v(\omega)+U(\omega),
\]
where \(U\) is the stochastic convolution from
\autoref{lem:stochastic-U-additive}. Let \(\rX, \mathrm{Z}, J\) be the corresponding
stochastic Lagrangian quantities. Choose $\delta_0  = \delta_0 (\Omega_0, q) >0 $ such that
\[
\delta_0\le \eps_*,
\]
where \(\eps_*\) is the constant from
\autoref{lem:aux-space-lagrangian}(iv). Fix
\(\theta\in(0,\frac12)\). For \(\delta\in(0,\delta_0]\), define
\[
\sigma_{\delta,\theta}
:=
\inf\bigl\{
t\in[0,T]:
\|\nabla \rX-I_3\|_{\rL^\infty(0,t;\rH^{1,q}(\Omega_0))}
+
\|\mathrm Z-I_3\|_{\rH^{\theta,p}(0,t;\rH^{1,q}(\Omega_0))}
+
\|J-1\|_{\rH^{\theta,p}(0,t;\rH^{1,q}(\Omega_0))}
\ge \delta
\bigr\}\wedge T.
\]
Then the following assertions hold.

\begin{enumerate}
\item[(i)] The flow satisfies
\[
\rX(\omega)-\mathrm{id}
\in
\rC\bigl([0,T];\rH^{2,q}(\Omega_0;\R^3)\bigr).
\]

\item[(ii)]  For every
\(s\in[0,\sigma_{\delta,\theta}(\omega))\),
the map \(\rX(\omega,s,\cdot)\) is a
\(\rC^1\)-diffeomorphism onto its image.
Further, there exists
\(C_0=C_0(\Omega_0,q,\delta_0)>0\)
such that for every
\(t\in(0,\sigma_{\delta,\theta}(\omega))\),
we have
\[
\|\mathrm Z(\omega)-I_3\|_{\rL^\infty(0,t;\rH^{1,q}(\Omega_0))}
+
\|J(\omega)-1\|_{\rL^\infty(0,t;\rH^{1,q}(\Omega_0))}
\le
C_0\delta,
\]
and
\[
\|\mathrm Z(\omega)\|_{\rL^\infty(0,t;\rH^{1,q}(\Omega_0))}
+
\|J(\omega)\|_{\rL^\infty(0,t;\rH^{1,q}(\Omega_0))}
\le
C_0.
\]
Moreover,
\[
\|\mathrm Z(\omega)-I_3\|_{\rH^{\theta,p}(0,t;\rH^{1,q}(\Omega_0))}
+
\|J(\omega)-1\|_{\rH^{\theta,p}(0,t;\rH^{1,q}(\Omega_0))}
\le \delta
\]
for every \(t\in[0,\sigma_{\delta,\theta}(\omega))\).

\item[(iii)] Assume that
\(
\|v(\omega) \|_{ {    \mathbb E_{1,T} }  } \leq R.
\)
Then there exists
\(
T_{\delta,R,\theta}(\omega)>0,
\)
depending only on \(\delta\), \(R\), \(\theta\), \(p\), \(q\), \(\Omega_0\), and
on the fixed pathwise quantities \(U(\omega)\) and \(\psi(\omega)\), such that
\[
\sigma_{\delta,\theta}(\omega)\ge T_{\delta,R,\theta}(\omega).
\]

\item[(iv)] Assume now that \(v\), \(U\), \(\rX\), \(\mathrm Z\), and \(J\) are
adapted processes and that \(\|v \|_{ {   \mathbb E_{1,T}}     } \le R\) a.s. Then
\(\sigma_{\delta,\theta}\) is an \((\mathcal F_t)_t\)-stopping time. Moreover,
there exists an \((\mathcal F_t)_t\)-stopping time
\(
\tau_{\delta,R,\theta}>0
\text{ a.s.},
\)
depending only on \(\delta\), \(R\), \(\theta\), \(p\), \(q\), \(\Omega_0\), and
on the adapted processes \(U\) and \(\psi\), such that
\[
\tau_{\delta,R,\theta}\le \sigma_{\delta,\theta}
\ \text{ a.s.}
\]
\end{enumerate}
\end{prop}

\begin{proof}
We work pathwise and suppress \(\omega\) from the notation.

\smallskip

\noindent
\emph{(i).}
Let \(\psi\) be the noise-only flow generated by
\[
d\psi_t(x)=\sum_{k=1}^K Q_k(\psi_t(x))\circ dW_t^k,
\quad
\psi_0(x)=x.
\]
Since \(Q_k\in \rC_b^\infty(\R^3;\R^3)\), Kunita's theory, \cite[Section~4]{Kun_90} gives a stochastic
flow of \(\rC^\infty\)-diffeomorphisms. Define
\[
\mathrm Y_t:=\psi_t^{-1}\circ \rX(t,\cdot)
\ \text{ so that }\
\rX(t,\cdot)=\psi_t\circ \mathrm Y_t.
\]
By the Stratonovich chain rule,
\begin{equation}
\label{eq:Y-random-ODE-rewritten}
\mathrm Y_t(y)
=
y+\int_0^t D\psi_s(\mathrm Y_s(y))^{-1}\,\bar u(s,y)\,ds.
\end{equation}
Applying \autoref{lem:aux-flow-lagrangian} to $
a_s:=(D\psi_s)^{-1},
b(s):=\bar u(s),$  we obtain
\[
\mathrm Y-\mathrm{id}
\in
\rC([0,T];\rH^{2,q}(\Omega_0;\R^3)) \ \text{ and } \
\partial_t\mathrm Y
\in
\rL^p(0,T;\rH^{2,q}(\Omega_0;\R^3)).
\]
Since \(t\mapsto\psi_t\begingroup -\endgroup \begingroup \mathrm{id}\endgroup \)  is continuous with values in \(\rC_b^3(\R^3)\),
\autoref{lem:aux-space-lagrangian}(ii) implies
\[
\rX-\mathrm{id}
=
(\psi-\mathrm{id})\circ \mathrm Y + (\mathrm Y-\mathrm{id})
\in
\rC([0,T];\rH^{2,q}(\Omega_0;\R^3)).
\]

\smallskip

\noindent
\emph{(ii).}
Let \(t<\sigma_{\delta,\theta}\). By definition of \(\sigma_{\delta,\theta}\),
\[
\|\nabla \rX-I_3\|_{\rL^\infty(0,t;\rH^{1,q}(\Omega_0))}<\delta\le \delta_0\le \eps_*.
\]
Hence, for every \(s\in[0,t]\),
\begin{equation}
\|\nabla \rX(s)-I_3\|_{\rH^{1,q}(\Omega_0)}\le \eps_*.
    \label{eq: nabla X everywhere invertible}
\end{equation}
Applying \autoref{lem:aux-space-lagrangian}(iv) with
\(
A=\nabla \rX(s,\cdot),
\,
B=I_3,
\)
we obtain
\[
\|\mathrm Z(s)-I_3\|_{\rH^{1,q}(\Omega_0)}
+
\|J(s)-1\|_{\rH^{1,q}(\Omega_0)}
\le
C_0\|\nabla \rX(s)-I_3\|_{\rH^{1,q}(\Omega_0)}.
\]
Taking the supremum over \(s\in[0,t]\), we get
\[
\|\mathrm Z-I_3\|_{\rL^\infty(0,t;\rH^{1,q}(\Omega_0))}
+
\|J-1\|_{\rL^\infty(0,t;\rH^{1,q}(\Omega_0))}
\le
C_0\delta.
\]

To verify global injectivity, set
\(h:=\rX(s,\cdot)-\mathrm{id}\).
Let \(h_{\Omega_0}:= \int_{\Omega_0} h \d x \) and let \(\widetilde h\in\rH^{2,q}(\R^3)\)
be a bounded extension of \(h-h_{\Omega_0}\).
By Poincar\'e's inequality and the embedding
\(\rH^{2,q}(\R^3)\hookrightarrow\rW^{1,\infty}(\R^3)\),
we have
\[
\|\nabla\widetilde h\|_{\rL^\infty(\R^3)}
\le C\|\widetilde h\|_{\rH^{2,q}(\R^3)}
\le C\|h-h_{\Omega_0}\|_{\rH^{2,q}(\Omega_0)}
\le C\|\nabla h\|_{\rH^{1,q}(\Omega_0)}.
\]
Since \(\widetilde h=h-h_{\Omega_0}\) on \(\Omega_0\),
the mean value theorem gives
\[
|h(x)-h(y)|
\le \|\nabla\widetilde h\|_{\rL^\infty(\R^3)}
|x-y|
\le C\|\nabla h\|_{\rH^{1,q}(\Omega_0)}|x-y|
\]
for all \(x,y\in\overline{\Omega}_0\).
Consequently, by the definition of
\(\sigma_{\delta,\theta}\) and the sufficiently
small choice of \(\delta_0=\delta_0(\Omega_0,q)\), we have
\[
|h(x)-h(y)|
\le C\delta_0|x-y|
\le \frac12|x-y|.
\]
Hence $ | \rX( s , x) -\rX(s,y)|
\ge\frac12|x-y|,$ so \( \rX(s,\cdot)\) is injective. Since \(\nabla\rX(s,\cdot)\) is everywhere
invertible thanks to \eqref{eq: nabla X everywhere invertible}, then \(\rX(s,\cdot)\) is a
\(\rC^1\)-diffeomorphism onto its image. The bound for \(\mathrm Z\) and \(J\) themselves follows by adding the norms of
\(I_3\) and \(1\). The \(\rH^{\theta,p}\)-bound is part of the definition of
\(\sigma_{\delta,\theta}\).

\smallskip

\noindent
\emph{(iii).} 
Fix
\(\alpha\in(\max\{\theta,1/p\},1/2)\)
and
\( \eta\in(\max\{\theta,1/p\},\alpha) \). Define
\[
\Lambda(t)
:=
1+\sup_{0\le s\le t}
\bigl(
\|D\psi_s\|_{\rC_b^3(\R^3)}
+
\| (  D\psi_s )^{-1} \|_{\rC_b^3(\R^3)}
\bigr), \quad 
\rho(t)
:=
\sup_{0\le s\le t}
\|D\psi_s-I_3\|_{\rC_b^2(\R^3)},
\] 
and
\[
K_\alpha(t)
:=
\sup_{0\le r<s\le t}
\frac{
\|D\psi_s-D\psi_r\|_{\rC_b^2(\R^3)}
+
\|(D\psi_s)^{-1}-(D\psi_r)^{-1}\|_{\rC_b^2(\R^3)}
}{|s-r|^\alpha}  .
\] 
Fix \(\beta\in(\alpha,\frac12)\). By Kunita \cite[Theorem 4.6.5]{Kun_90},
the stochastic flow \((\psi_{s,t})_{0\le s\le t\le T}\) is a flow of
\(\rC^\infty\)-diffeomorphisms (even more regular). Moreover, by \cite[Theorem 4.6.4]{Kun_90}, for each multi-index
\(|\gamma|\le 3\), the derivatives \(D_x^\gamma \psi_{s,t}(x)\) and
\(D_x^\gamma \psi_{s,t}^{-1}(x)\) satisfy suitable moment estimates for increments
in the variables \((s,t,x)\). Applying Kunita's Kolmogorov continuity criterion
\cite[Theorem 1.4.1]{Kun_90}, we obtain that, $t\mapsto\psi_t(\omega) -   \mathrm{id} $  and $t\mapsto\psi_t^{-1}(\omega)- \mathrm{id}  $  are in $
\rC^\beta([0,T];\rC_b^3(\R^3;\R^3))$, with probability one. Therefore, for almost every \(\omega\), there exists \(C_\beta(\omega)<\infty\) such that
\[
\|D\psi_s(\omega)-D\psi_r(\omega)\|_{\rC_b^2(\R^3)}
+
\| ( D\psi_s)^{-1} (\omega)- ( D\psi_r )^{-1} (\omega)\|_{\rC_b^2(\R^3)}
\le
C_\beta(\omega)\,|s-r|^\beta
\] 
for all \(r,s\in[0,T]\). Since \(\beta>\alpha\), the map
\[
\Phi_\omega(r,s)
:=
\begin{cases}
\displaystyle
\frac{
\|D\psi_s(\omega)-D\psi_r(\omega)\|_{\rC_b^2}
+
\|(D\psi_s)^{-1}(\omega)-(D\psi_r)^{-1}(\omega)\|_{\rC_b^2}
}{|s-r|^\alpha}  ,
& 0\le r<s\le T,\\[2ex]
0,& r=s,
\end{cases}
\] 
extends continuously to \([0,T]^2\). Hence $
K_\alpha(t,\omega)
=
\max_{0\le r\le s\le t}\Phi_\omega(r,s)$ is finite, pathwise continuous, and nondecreasing in \(t\).
 Set $
B_R(t):=R+\|U\|_{\rL^p(0,t;\rH^{2,q}(\Omega_0))}.$ Applying \autoref{lem:aux-flow-lagrangian} to \eqref{eq:Y-random-ODE-rewritten}
gives
\[
\|\mathrm Y-\mathrm{id}\|_{\rL^\infty(0,t;\rH^{2,q}(\Omega_0))}
\le
\beta_R(t)\bigl(1+\|\mathrm Y-\mathrm{id}\|_{\rL^\infty(0,t;\rH^{2,q}(\Omega_0))}
+\|\mathrm Y-\mathrm{id}\|_{\rL^\infty(0,t;\rH^{2,q}(\Omega_0))}^2\bigr),
\]
where $\beta_R(t):=C\,\Lambda(t)\,t^{1-\frac1p}B_R(t).$ Define
\[
T_{1,R}(\omega)
:=
\inf\bigl\{t\in[0,T]:\beta _R(t)\ge \tfrac18\bigr\}\wedge T.
\]
Then \(T_{1,R}(\omega)>0\), and for \(t<T_{1,R}(\omega)\), by \eqref{eq:aux-flow-smalltime} we have
\begin{equation}
\label{eq:Y-Linfty-final}
\|\mathrm Y-\mathrm{id}\|_{\rL^\infty(0,t;\rH^{2,q}(\Omega_0))}
\le
M_0(t)
:=
7\beta _R(t).
\end{equation}
Using \eqref{eq:aux-flow-derivative} and \eqref{eq:aux-flow-fractional}, we obtain
\begin{equation}
\label{eq:Y-Htheta-final}
\|\mathrm Y-\mathrm{id}\|_{\rH^{\theta,p}(0,t;\rH^{2,q}(\Omega_0))}
\le
M_\theta(t)
:=
C\,t^{1-\theta}\Lambda(t)\bigl(1+\beta _R(t)\bigr)B_R(t).
\end{equation}
In particular, \(M_0(t)\to0\) and \(M_\theta(t)\to0\) as \(t\downarrow0\).
We now estimate \(\nabla \rX-I_3\). Since $
\nabla \rX(t)=D\psi_t(\mathrm Y_t)\,\nabla \mathrm Y_t,$ we can write
\[
\nabla \rX-I_3
=
\bigl(D\psi(\mathrm Y)-I_3\bigr)\nabla \mathrm Y
+
(\nabla \mathrm Y-I_3).
\]
Set $
B(t):=D\psi_t(\mathrm Y_t)-I_3.$ By \autoref{lem:aux-space-lagrangian}(ii), we obtain
\[
\|B\|_{\rL^\infty(0,t;\rH^{1,q}(\Omega_0))}
\le
C\,\rho(t)\bigl(1+M_0(t)+M_0(t)^2\bigr).
\]
Hence, by this previous bound and \eqref{eq:Y-Linfty-final}, we get
\begin{equation}
\label{eq:A0-final}
\begin{aligned}
    \|\nabla \rX-I_3\|_{\rL^\infty(0,t;\rH^{1,q}(\Omega_0))} &\le
C\,\|B\|_{\rL^\infty(0,t;\rH^{1,q}(\Omega_0))}\bigl(1+M_0(t)\bigr)+M_0(t)
\\&\le
C\bigl(
\rho(t)\bigl(1+M_0(t)+M_0(t)^2\bigr)\bigl(1+M_0(t)\bigr)
+
M_0(t) 
\bigr)=:A_0(t).
\end{aligned}
\end{equation}
Again \(A_0(t)\to0\) as \(t\downarrow0\).
Next we estimate the fractional-time norm of \(B\). For \(0\le s<r\le t\),
\[
B(r)-B(s)
=
\bigl(D\psi_r(\mathrm Y_r)-D\psi_r(\mathrm Y_s)\bigr)
+
\bigl((D\psi_r-D\psi_s)(\mathrm Y_s)\bigr).
\]
By \autoref{lem:aux-space-lagrangian}(iii),
\[
\|D\psi_r(\mathrm Y_r)-D\psi_r(\mathrm Y_s)\|_{\rH^{1,q}(\Omega_0)}
\le
C\,\Lambda(t)\bigl(1+M_0(t)\bigr)^2
\|\mathrm Y_r-\mathrm Y_s\|_{\rH^{2,q}(\Omega_0)},
\]
while \autoref{lem:aux-space-lagrangian}(ii) gives
\[
\|(D\psi_r-D\psi_s)(\mathrm Y_s)\|_{\rH^{1,q}(\Omega_0)}
\le
C\,K_\alpha(t)\,|r-s|^\alpha\,
\bigl(1+M_0(t)+M_0(t)^2\bigr).
\]
{ 
Using the preceding increment bounds and
\(\partial_t\mathrm Y\in
\rL^p(0,t;\rH^{2,q}(\Omega_0))\),
we first estimate the scaled
\(\rW^{\eta,p}\)-norm of \(B\).
Since \(B(0)=0\), the embedding
\eqref{eq:aux-time-embedding} in the proof of \autoref{lem:aux-time-inverse-det} then gives

\begin{equation}
\label{eq:Btheta-final}
\begin{aligned}
\|B\|_{\rH^{\theta,p}(0,t;\rH^{1,q}(\Omega_0))}
+t^{\eta-\theta}
\|B\|_{\rW^{\eta,p}(0,t;\rH^{1,q}(\Omega_0))}
 &\le
C\bigl(
\Lambda(t)\bigl(1+M_0(t)\bigr)^2M_\theta(t)
+
K_\alpha(t)\,t^{\alpha-\theta}
\bigl(1+M_0(t)+M_0(t)^2\bigr)
\\&\qquad +
t^{1/p}\rho(t)\bigl(1+M_0(t)+M_0(t)^2\bigr)
\bigr) =:B_\theta(t) .
\end{aligned}
\end{equation} 
Since \(\partial_t\mathrm Y\in
\rL^p(0,t;\rH^{2,q}(\Omega_0))\)
and \(\mathrm Y_0=\mathrm{id}\), we also have $
t^{\eta-\theta}
\|\nabla\mathrm Y-I_3\|_{
\rW^{\eta,p}(0,t;\rH^{1,q}(\Omega_0))}
\le CM_\theta(t).$ Using the Banach algebra property of
\(\rH^{1,q}(\Omega_0)\), the
difference-quotient product estimate in
\(\rW^{\eta,p}\), and
\eqref{eq:aux-time-embedding}, we obtain
\[
\nabla \rX-I_3
=
B\,\nabla \mathrm Y + (\nabla \mathrm Y-I_3)
\in
\rH^{\theta,p}(0,t;\rH^{1,q}(\Omega_0)),
\]
and
\begin{equation}
\label{eq:Atheta-final}
\begin{aligned}
\|\nabla\rX-I_3\|_{
\rH^{\theta,p}(0,t;\rH^{1,q}(\Omega_0))}
&
+t^{\eta-\theta}
\|\nabla\rX-I_3\|_{
\rW^{\eta,p}(0,t;\rH^{1,q}(\Omega_0))}
\le
C\bigl[
B_\theta(t)\bigl(1+M_0(t)\bigr)
\\
&\quad + C \rho (t) \bigl( 1 + M_0 (t ) + M_0 ( t)^2 \bigr) M_\theta(t)
+
M_\theta(t)
\bigr]=: A_\theta(t).
\end{aligned}
\end{equation} 
}
Again \(A_\theta(t)\to0\) as \(t\downarrow0\).
Now choose \(t<T_{1,R}(\omega)\) such that $ A_0(t)\le \delta_0\le \eps_*.$ Then
\[
\|\nabla \rX-I_3\|_{\rL^\infty(0,t;\rH^{1,q}(\Omega_0))}
\le
A_0(t)
\le
\eps_*.
\]
{ 
Applying \autoref{lem:aux-time-inverse-det}
to \(A:=\nabla\rX\), and using \eqref{eq:Atheta-final}, we arrive at
}
\begin{equation}
\label{eq:ZJ-estimate-final}
\|\mathrm Z-I_3\|_{\rH^{\theta,p}(0,t;\rH^{1,q}(\Omega_0))}
+
\|J-1\|_{\rH^{\theta,p}(0,t;\rH^{1,q}(\Omega_0))}
\le
C_0\bigl(A_\theta(t)+t^{1/p}A_0(t)\bigr),
\end{equation}
and
\[
\|\mathrm Z-I_3\|_{\rL^\infty(0,t;\rH^{1,q}(\Omega_0))}
+
\|J-1\|_{\rL^\infty(0,t;\rH^{1,q}(\Omega_0))}
\le
C_0A_0(t).
\]
Finally define
\[
G_{R,\theta}(t)
:=
A_0(t)+C_0\bigl(A_\theta(t)+t^{1/p}A_0(t)\bigr),
\]
and
\[
T_{\delta,R,\theta}(\omega)
:=
T_{1,R}(\omega)\wedge
\inf\bigl\{t\in[0,T]:G_{R,\theta}(t)\ge \delta\bigr\}\wedge T.
\]
Since \(G_{R,\theta}(0)=0\) and \(G_{R,\theta}\) is continuous and
nondecreasing, we have \(T_{\delta,R,\theta}(\omega)>0\). If
\(t<T_{\delta,R,\theta}(\omega)\), then
\[
\|\nabla \rX-I_3\|_{\rL^\infty(0,t;\rH^{1,q}(\Omega_0))}
+
\|\mathrm Z-I_3\|_{\rH^{\theta,p}(0,t;\rH^{1,q}(\Omega_0))}
+
\|J-1\|_{\rH^{\theta,p}(0,t;\rH^{1,q}(\Omega_0))}
<
\delta.
\]
By the definition of \(\sigma_{\delta,\theta}\), this implies
\[
T_{\delta,R,\theta}(\omega)\le \sigma_{\delta,\theta}(\omega).
\]

\smallskip

\noindent
\emph{(iv).} 
Let $ \zeta:= \inf \{
t\in[0,T]: \| \nabla\rX-I_3\|_{
\rL^\infty(0,t;\rH^{1,q}(\Omega_0))}
\geq \delta_0  \} \wedge T.$ Since $\nabla \rX$ is adapted and continuous, then $\zeta$ is a stopping time. By \autoref{lem:aux-space-lagrangian}(iv), the quantities $\mathrm Z$ and $J$ are well-defined
on $[0,\zeta)$. For $t< \zeta,$ define
\[
\mathcal N(t)
:=
\|\nabla \rX-I_3\|_{\rL^\infty(0,t;\rH^{1,q}(\Omega_0))}
+
\|\mathrm Z-I_3\|_{\rH^{\theta,p}(0,t;\rH^{1,q}(\Omega_0))}
+
\|J-1\|_{\rH^{\theta,p}(0,t;\rH^{1,q}(\Omega_0))}.
\]
We set $\mathcal{N}(t) = + \infty$, if $ t \geq \zeta $. For each fixed \(t\), the random variable \(\mathcal N(t)\) is
\(\mathcal F_t\)-measurable. Indeed, since \(\rH^{1,q}(\Omega_0)\) is separable
and the trajectories are continuous,
\[
\|\nabla \rX-I_3\|_{\rL^\infty(0,t;\rH^{1,q}(\Omega_0))}
=
\sup_{r\in[0,t]\cap\mathbb Q}
\|\nabla \rX(r)-I_3\|_{\rH^{1,q}(\Omega_0)},
\]
which is \(\mathcal F_t\)-measurable. On every interval \([0,t]\) with \(t<\zeta\),
the regularity of \(\psi\) and \(\mathrm Y\),
together with the smoothness of the inverse
and determinant maps near \(I_3\), implies
\[
\nabla\rX,\ \mathrm Z,\ J
\in\rC^\beta([0,t];\rH^{1,q}(\Omega_0)),
\qquad
\max\{\theta,1/p\}<\beta<1/2.
\]
Consequently, their fractional Sobolev
restriction norms are adapted and continuous
in the interval endpoint. Indeed,
measurability follows from the adaptedness
and continuity of the trajectories, while
continuity of the norms follows from the
embedding
\(\rC^\beta\hookrightarrow\rH^{\theta,p}\).
The  restriction norms are also
non-decreasing. With the preceding convention
at \(\zeta\), the process \(\mathcal N\) is
therefore adapted, non-decreasing and
right-continuous. Hence, for \(t\in[0,T)\), we have
\[
\{\sigma_{\delta,\theta}\le t\}
=
\bigl\{\exists r\in[0,t]: \mathcal N(r)\ge \delta\bigr\}
=
\{\mathcal N(t)\ge \delta\}
\in \mathcal F_t.
\]
For \(t=T\), the claim is trivial since
\(\{\sigma_{\delta,\theta}\le T\}=\Omega\in\mathcal F_T\).
Therefore \(\sigma_{\delta,\theta}\) is an \((\mathcal F_t)_t\)-stopping time.
Now define the adapted, continuous, nondecreasing processes
\[
\beta_R(t):=C\,\Lambda(t)\,t^{1-\frac1p}B_R(t) \ \text{ and } \
G_{R,\theta}(t):=
A_0(t)+C_0\bigl(A_\theta(t)+t^{1/p}A_0(t)\bigr),
\]
where \(\Lambda,\rho,K_\alpha,B_R,M_0,M_\theta,A_0,B_\theta,A_\theta\) are
given by the same explicit formulas as in part \emph{(iii)}. Set
\[
\tau_{1,R}
:=
\inf\bigl\{t\in[0,T]:\beta_R(t)\ge \tfrac18\bigr\}\wedge T,
\]
and
\[
\tau_{\delta,R,\theta}
:=
\tau_{1,R}\wedge
\inf\bigl\{t\in[0,T]:G_{R,\theta}(t)\ge \delta\bigr\}\wedge T.
\]
Since these processes start from \(0\) and are continuous,
\[
\tau_{\delta,R,\theta}>0
\ \text{ a.s.}
\]
Finally, for almost every \(\omega\), part \emph{(iii)} yields
\[
\tau_{\delta,R,\theta}(\omega)\le \sigma_{\delta,\theta}(\omega).
\]
This completes the proof.
\end{proof}

\noindent
\noindent
We now introduce the unknowns in Lagrangian coordinates by composing the
Eulerian density and velocity with the stochastic flow map \(\rX\). More
precisely, for \((t,y)\in[0,T]\times\Omega_0\), we set
\begin{equation}
\label{eq:pullback-unknowns}
\bar{\varrho}(t,y)
:=
\varrho\bigl(t,\rX(t,y)\bigr) \ \text{ and }\ 
\bar{u}(t,y)
:=
u\bigl(t,\rX(t,y)\bigr).
\end{equation}
In other words, \(\bar\varrho\) and \(\bar u\) denote the density and velocity
in the reference coordinates associated with the initial domain \(\Omega_0\).
Hence the original free-boundary problem on the moving domain \(\Omega_t\) is
transformed into a system posed on the fixed reference domain \(\Omega_0\).
We next identify the transformed density in terms of the Jacobian of the
stochastic Lagrangian flow. Since the continuity equation and the flow are both
written in Stratonovich form, the usual chain rule applies. Thus,
we obtain
\[
\d\bar\varrho
=
-\bar\varrho\,
\operatorname{div}_x u(t,\rX(t,y))\,\d t
-
\bar\varrho
\sum_{k=1}^K
\operatorname{div}Q_k(\rX(t,y))\circ\d W_t^k .
\]
Using
\[
\operatorname{div}_x u(t,\rX(t,y))
=
\nabla\bar u(t,y):\mathrm Z(t,y)^\top,
\]
this becomes
\[
\d\bar\varrho
=
-\bar\varrho\,
\nabla\bar u:\mathrm Z^\top\,\d t
-
\bar\varrho
\sum_{k=1}^K
\operatorname{div}Q_k(\rX)\circ\d W_t^k .
\]
On the other hand, the Jacobian $J$
satisfies
\[
\d J
=
J\,\nabla\bar u:\mathrm Z^\top\,\d t
+
J
\sum_{k=1}^K
\operatorname{div}Q_k(\rX)\circ\d W_t^k
\ \text{ with }\
J(0,\cdot)=1.
\]
Since the transport vector fields are divergence-free, 
the stochastic terms in the last two identities vanish. Hence
\[
\partial_t\bar\varrho
+
\bar\varrho\,\nabla\bar u:\mathrm Z^\top
=0
\]
and
\[
\partial_t J
=
J\,\nabla\bar u:\mathrm Z^\top
\ \text{ with }\
J(0,\cdot)=1.
\]
Consequently,
\[
\partial_t(\bar\varrho J)=0
\ \text{ and therefore }\
\bar\varrho(t,y)J(t,y)=\varrho_0(y).
\]
In particular, the transformed density is completely determined by the initial
density and the Jacobian of the stochastic Lagrangian flow, namely
\begin{equation}
\label{eq:rhoJ}
\bar\varrho(t,y)=\frac{\varrho_0(y)}{J(t,y)}.
\end{equation}
Thus the transformed problem may be written only in terms of the velocity
\(\bar u\), with the density always understood through
\eqref{eq:rhoJ}.
In this way we obtain
\begin{equation}
\left\{
\begin{aligned}
\d \bar u
-\varrho_0^{-1}\operatorname{div}\mathbb S(\nabla \bar u)\,\d t
&=
F_u(\bar u)\,\d t
+
f_{\mathrm{sto}}\,\d B_t,
&& \text{in } \Omega_0\times(0,T),
\\[0.4em]
\mathbb S(\nabla \bar u)N
&=
F_\Gamma(\bar u),
&& \text{on } \Gamma_0\times(0,T),
\\[0.4em]
\bar u(0)&=u_0,
&& \text{in } \Omega_0.
\end{aligned}
\right.
\label{eq:transformed-FB-CNSE-lame-divrho}
\end{equation}
Here \(N\) denotes the outer unit normal vector on the fixed boundary
\(\Gamma_0\). The transformed coefficient is exactly the prescribed \(f_{\mathrm{sto}}(t,y)\), since \(\rF(t,\rX(t,y))=y\).

The nonlinear right-hand side in the velocity equation is given by
\begin{equation}
\label{eq:transformed-FB-CNSE-lame-divrho-F2}
\begin{aligned}
( F_u(\bar u))_i
&=
\frac{J-1}{\varrho_0}
\Bigl(
\mu\sum_k \frac{\partial^2 \bar u_i}{\partial y_k^2}
+
(\mu+\lambda)\sum_j \frac{\partial^2 \bar u_j}{\partial y_i\partial y_j}
\Bigr)
\\
&\quad
+
\frac{\mu J}{\varrho_0}\Bigl(
\sum_{j,k,l}
\frac{\partial^2 \bar u_i}{\partial y_k \partial y_l}
\bigl(\mathrm Z_{k,j}-\delta_{k,j}\bigr)\mathrm Z_{l,j}
+
\sum_{k,l}
\frac{\partial^2 \bar u_i}{\partial y_k \partial y_l}
\bigl(\mathrm Z_{l,k}-\delta_{l,k}\bigr)
+
\sum_{j,k,l}
\mathrm Z_{l,j}\frac{\partial \bar u_i}{\partial y_k}
\frac{\partial \mathrm Z_{k,j}}{\partial y_l}
\Bigr)
\\
&\quad
+
\frac{(\mu+\lambda)J}{\varrho_0}\Bigl(
\sum_{j,k,l}
\frac{\partial^2 \bar u_j}{\partial y_k \partial y_l}
\bigl(\mathrm Z_{k,j}-\delta_{k,j}\bigr)\mathrm Z_{l,i}
+
\sum_{j,l}
\frac{\partial^2 \bar u_j}{\partial y_j \partial y_l}
\bigl(\mathrm Z_{l,i}-\delta_{l,i}\bigr)
+
\sum_{j,k,l}
\mathrm Z_{l,i}\frac{\partial \bar u_j}{\partial y_k}
\frac{\partial \mathrm Z_{k,j}}{\partial y_l}
\Bigr)
\\
&\quad
-
\frac{J}{\varrho_0}
\sum_j
\mathrm Z_{j,i}
\frac{\partial}{\partial y_j}
p\left(\frac{\varrho_0}{J}\right).
\end{aligned}
\end{equation}
The boundary nonlinearity is
\begin{equation}
\label{eq:transformed-FB-CNSE-lame-divrho-FGamma}
\begin{aligned}
(F_\Gamma(\bar u))_i
&=
\mu\sum_j \frac{\partial \bar u_i}{\partial y_j}
\Bigl(N_j-\sum_l J\,\mathrm Z_{l,j}N_l\Bigr)
+
\mu\sum_j \frac{\partial \bar u_j}{\partial y_i}
\Bigl(N_j-\sum_l J\,\mathrm Z_{l,j}N_l\Bigr)
\\
&\quad
+
\lambda\sum_k \frac{\partial \bar u_k}{\partial y_k}
\Bigl(N_i-\sum_m J\,\mathrm Z_{m,i}N_m\Bigr)
+
\mu\sum_{j,k,l}
\Bigl(\delta_{k,j}-\mathrm Z_{k,j}\Bigr)
\frac{\partial \bar u_i}{\partial y_k}\,
J\,\mathrm Z_{l,j}N_l
\\
&\quad
+
\mu\sum_{j,k,l}
\Bigl(\delta_{k,i}-\mathrm Z_{k,i}\Bigr)
\frac{\partial \bar u_j}{\partial y_k}\,
J\,\mathrm Z_{l,j}N_l
+
\lambda\Bigl(
\sum_k \frac{\partial \bar u_k}{\partial y_k}
-
\sum_{k,l}\mathrm Z_{l,k}\frac{\partial \bar u_k}{\partial y_l}
\Bigr)
\sum_m J\,\mathrm Z_{m,i}N_m
\\
&\quad
+
\left(p\left(\frac{\varrho_0}{J}\right)-p_{\mathrm{ext}}\right)
\sum_j J\,\mathrm Z_{j,i}N_j .
\end{aligned}
\end{equation}
We conclude this section by deriving a difference estimate for two stochastic Lagrangian transforms \(\mathrm Z_1\) and \(\mathrm Z_2\) associated with two velocity fields. This estimate will play an essential role in the contraction mapping argument used later to prove local existence for the transformed system.

\begin{cor}[Pathwise difference estimates for the stochastic Lagrangian transform]
\label{cor:pathwise-difference-stochastic-lagrangian}\mbox{} \\
Under the same assumptions and notation as in
\autoref{prop:pathwise-stochastic-lagrangian}, let
\[
v_1(\omega),\,v_2(\omega)\in  \mathbb E_{1,T}  \ \text{ with } \
\|v_1(\omega \|_{\mathbb E_{1,T}}  ,\,
\|v_2(\omega) \|_{\mathbb E_{1,T}}  \le R,
\]
and set
\[
\bar u_i(\omega):=v_i(\omega)+U(\omega) \ \text{ for }\  i=1,2,
\]
where \(U\) denotes the stochastic convolution from
\autoref{lem:stochastic-U-additive}. Let
\(
\rX_i,\, \mathrm Z_i,\, J_i
\text{ for } i=1,2,
\)
be the associated stochastic Lagrangian quantities from
\autoref{prop:pathwise-stochastic-lagrangian}. Then, on \([0,\tau(\omega)]\), where \(\tau=\tau_{\delta,R,\theta}\), there exists a constant \(C(\omega)>0\), depending only on \(p\), \(q\), \(\eps\), \(\Omega_0\), \(\delta_0\), \(R\), and on the fixed pathwise quantities \(U(\omega)\) and \(\psi(\omega)\), such that
\begin{equation*}
\begin{aligned}
\|\mathrm Z_1-\mathrm Z_2\|_{\rL^\infty(0,\tau;\rH^{1,q}(\Omega_0))}
+
\|J_1-J_2\|_{\rL^\infty(0,\tau;\rH^{1,q}(\Omega_0))}
&\le
C(\omega)\,\tau^{1-\nicefrac1p}
\|v_1-v_2 \|_{\mathbb E_{1,\tau}} ,
\\
\|\mathrm Z_1-\mathrm Z_2\|_{\rH^{\theta,p}(0,\tau;\rH^{1,q}(\Omega_0))}
+
\|J_1-J_2\|_{\rH^{\theta,p}(0,\tau;\rH^{1,q}(\Omega_0))}
&\le
C(\omega)\,\tau^{\alpha-\theta}
\|v_1-v_2 \|_{\mathbb E_{1,\tau}}  ,
\end{aligned}
\end{equation*}
where \(\alpha\in(\theta,\nicefrac12)\).
\end{cor}

\section{The linearized transformed system}
\noindent
In this section, we analyze the linearized system associated with the transformed free-boundary problem \eqref{eq:transformed-FB-CNSE-lame-divrho}. 
Our aim is to separate the deterministic optimal-data part from the stochastic additive forcing.
To this end, we fix \(\omega\in\Omega\). All objects below are therefore understood pathwise at the sample \(\omega\), and we display the dependence on \(\omega\) explicitly at first. We begin with the linearized transformed velocity system
\begin{equation}
\left\{
\begin{aligned}
\d\bar u(\omega) -\varrho_0(\omega )^{-1} \operatorname{div}\mathbb S(\nabla\bar u(\omega) )\,\d t &= f_u(\omega)\,\d t + f_{\mathrm{sto}}(\omega)\,\d B_t(\omega),
&& \text{in } \Omega_0\times(0,T),
\\
\mathbb S(\nabla \bar u(\omega))N &= g(\omega),
&& \text{on } \Gamma_0\times(0,T),
\\
\bar u(\omega,0) &= u_0(\omega),
&& \text{in } \Omega_0,
\end{aligned}
\right.
\label{eq:linearized-transformed-system}
\end{equation}
where \(f_u(\omega)\), \(f_{\mathrm{sto}}(\omega)\), and \(g(\omega)\) are regarded as given.
The density is recovered from the stochastic Lagrangian flow. More precisely, once \(\bar u\) is known, the corresponding flow \(\rX\), inverse gradient \(\mathrm Z\), and Jacobian \(J\) are determined, and the transformed density is given by
\[
\bar\varrho(t,y)=\frac{\varrho_0(y)}{J(t,y)}.
\]
To analyze the velocity equation, we introduce the differential operators
\begin{equation}
\label{eq:def-calA}
(\mathcal A(\omega,x,D)u)_i
:=
-\frac{1}{\varrho_0(\omega,x)}
\sum_{j=1}^3
\partial_j
\Bigl(
\mu(\partial_j u_i+\partial_i u_j)
+\lambda\,\delta_{ij}\operatorname{div}u
\Bigr) \ \text{ for } \  i=1,2,3,
\end{equation}
that is,
\[
\mathcal A(\omega,x,D)u
=
-\varrho_0(\omega,x)^{-1}\operatorname{div}\mathbb S(\nabla u)
=
-\frac{\mu}{\varrho_0(\omega,x)}\Delta u
-\frac{\mu+\lambda}{\varrho_0(\omega,x)}\nabla\operatorname{div}u,
\]
and
\begin{equation}
\label{eq:def-calB}
(\mathcal B(x,D)u)_i
:=
\sum_{j=1}^3
\Bigl(
\mu(\partial_j u_i+\partial_i u_j)
+\lambda\,\delta_{ij}\operatorname{div}u
\Bigr)N_j \ \text{ for } \  i=1,2,3,
\end{equation}
i.e.
\[
\mathcal B(x,D)u = \mathbb S(\nabla u)N.
\]
For notational convenience, in the remainder of this section we suppress the dependence on \(\omega\), but all objects and arguments are understood pathwise for the fixed sample \(\omega\in\Omega\).
We split the velocity into
\[
\bar u = v + U,
\]
where \(v\) captures the deterministic inhomogeneous boundary data and \(U\) captures the stochastic forcing.
More precisely, \(v\) solves pathwise the deterministic Lam\'e system
\begin{equation}
\left\{
\begin{aligned}
\partial_t v + \mathcal A(x,D)v &= f_u,
&& \text{in } \Omega_0\times(0,T),
\\
\mathcal B(x,D)v &= g,
&& \text{on } \Gamma_0\times(0,T),
\\
v(0) &= u_0,
&& \text{in } \Omega_0,
\end{aligned}
\right.
\label{eq:v-system}
\end{equation}
while \(U\) solves the stochastic problem with homogeneous boundary conditions
\begin{equation}
\left\{
\begin{aligned}
\d U + A_q U\,\d t &= f_{\mathrm{sto}}\,\d B_t,
&& \text{in } (0,T),
\\
U(0) &= 0.
\end{aligned}
\right.
\label{eq:U-system}
\end{equation}
Here \(A_q=A_q(\omega)\) denotes the realization of \(\mathcal A(\omega,x,D)\) in \(\rLq(\Omega_0;\R^3)\) with homogeneous boundary condition \(\mathcal B(x,D)u=0\), i.e.
\[
A_q u := \mathcal A(\omega,x,D)u,
\quad
\rD(A_q)
:=
\bigl\{
u\in \rH^{2,q}(\Omega_0;\R^3): \mathcal B(x,D)u=0 \text{ on } \Gamma_0
\bigr\}.
\]

\begin{asu}\label{asu:v-data}
Let \(p>2\), \(q>3\), and \(2/p+3/q<1\).
Given \(f_u\), \(u_0\), and \(g\), assume
\begin{enumerate}[(i)]
    \item
    \(
    f_u \in \rLp(0,T;\rLq(\Omega_0;\R^3))
    \eqqcolon \mathbb E_{0,T},
    \)
    \item
    \(
    u_0 \in \rB^{2- \nicefrac{2}{p}}_{qp}(\Omega_0 ;\R^3))=: \rX_\gamma,
    \)
    \item
    \(
    g \in \rF_{pq}^{ \nicefrac12-\nicefrac{1}{2q}}(0,T;\rLq(\Gamma_0;\R^3))
    \cap
    \rLp(0,T;\rB_{qq}^{1- \nicefrac{1}{q}}(\Gamma_0;\R^3))
    \eqqcolon \mathbb G_T,
    \)
    \item if \(\frac{1}{2} - \frac{1}{2q} > \frac{1}{p}\), then
    \(
    \mathcal B(x,D)u_0 = g(0)
     \text{ on } \Gamma_0.
    \)
\end{enumerate}
\end{asu}
\noindent
The next result provides the optimal data theory for the deterministic Lam\'e problem which is inspired by \cite{DHP:07}.

\begin{lem}\label{lem:deterministic-v}
Let \(p \in (2,\infty)\), $q\in (3,\infty)$ satisfy \(2/p+3/q<1\), and let \(\Omega_0\subset\R^3\) be a bounded domain with boundary of class \(\rC^2\), and assume
\[
\varrho_0 \in \rH^{1,q}(\Omega_0) \ \text{ and } \
0<\varrho_\ast \le \varrho_0(x) \le \varrho^\ast < \infty
\quad \text{for all } x\in\overline{\Omega}_0.
\]
Then the problem \eqref{eq:v-system} admits a unique strong solution
\[
v \in \mathbb E_{1,T}
=
\rH^{1,p}(0,T;\rLq(\Omega_0;\R^3))
\cap
\rLp(0,T;{   \rH^{2,q}(\Omega_0, \R^3)  })
\]
if and only if the data \((f_u,u_0,g)\) satisfy \autoref{asu:v-data}.
In that case there exists a constant \(C>0\) such that
\[
\|v\|_{\mathbb E_{1,T}}
\le
C\bigl(
\|f_u\|_{\mathbb E_{0,T}}
+
\|u_0\|_{ \rX_{\gamma}}
+
\|g\|_{\mathbb G_T}
\bigr).
\]
\end{lem}

\begin{proof}
We verify the assumptions of the optimal \( \rLp \)-\( \rLq \)-theory for parabolic boundary value problems with inhomogeneous boundary data, see \cite{DHP:07}.
\begin{step}[Verification of the coefficient assumptions] \mbox{} \\
Since \(q>3\), Sobolev's embedding theorem yields
\[
\rH^{1,q}(\Omega_0)\hookrightarrow \rC(\overline{\Omega}_0).
\]
Hence \(\varrho_0\) admits a continuous representative on \(\overline{\Omega}_0\), and therefore so does \(\varrho_0^{-1}\). In particular, the coefficients of the principal part of \(\mathcal A(x,D)\) are continuous on \(\overline{\Omega}_0\). Moreover, \(\mathcal A(x,D)\) has no lower-order terms, so the interior smoothness assumptions are trivially satisfied.
Likewise, the coefficients of the boundary operator \(\mathcal B(x,D)\) are determined by the outer unit normal \(N\). Since \(\Gamma_0\) is of class \(\rC^2\), we have \(N\in \rC^1(\Gamma_0;\R^3)\), and hence the coefficients of the first-order boundary operator \(\mathcal B(x,D)\) are continuous on \(\Gamma_0\). Thus the boundary smoothness assumptions are satisfied as well.
\end{step}
\begin{step}[Parameter-ellipticity of the principal symbol]\mbox{} \\
The principal symbol of \(\mathcal A(x,D)\) is given by
\[
\mathcal A^\#(x,\xi)
=
\frac{1}{\varrho_0(x)}
\Bigl(
\mu |\xi|^2 I_3
+
(\mu+\lambda)\,\xi\otimes\xi
\Bigr),
\quad x\in \overline{\Omega}_0,\ \xi\in \R^3.
\]
Let \(a\in \R^3\). Then
\[
\mathcal A^\#(x,\xi)a
=
\frac{\mu}{\varrho_0(x)}|\xi|^2 a
+
\frac{\mu+\lambda}{\varrho_0(x)}\,\xi(\xi\cdot a).
\]
Hence, if \(a\perp \xi\), then
\[
\mathcal A^\#(x,\xi)a
=
\frac{\mu}{\varrho_0(x)}|\xi|^2 a,
\]
whereas for \(a=\xi\) one obtains
\[
\mathcal A^\#(x,\xi)\xi
=
\frac{2\mu+\lambda}{\varrho_0(x)}|\xi|^2 \xi.
\]
Therefore the eigenvalues of \(\mathcal A^\#(x,\xi)\) are
\[
\frac{\mu}{\varrho_0(x)}|\xi|^2
\quad \text{with multiplicity }2 \ \text{ and } \
\frac{2\mu+\lambda}{\varrho_0(x)}|\xi|^2
\quad \text{with multiplicity }1.
\]
Since \(\mu>0\), \(2\mu+\lambda>0\), and \(\varrho_0(x)\ge \varrho_\ast>0\), all eigenvalues are strictly positive for \(\xi\neq 0\). In particular,
\[
\sigma(\mathcal A^\#(x,\xi)) \subset (0,\infty)
\subset \C_+,
\quad x\in \overline{\Omega}_0,\ \xi\neq 0.
\]
Thus \(\mathcal A(x,D)\) is normally elliptic. In fact, the symbol is parameter-elliptic of angle \(0\).
\end{step}
\begin{step}[Lopatinskii--Shapiro condition]\mbox{} \\
Fix \(x_0\in \Gamma_0\), let \(\nu:=N(x_0)\), let \(\xi\in \R^3\) satisfy \(\xi\cdot \nu=0\), and let \(\eta\in \C_+\) with \(|\xi|+|\eta|\neq 0\). We consider the half-line problem
\begin{equation}
\label{eq:LS-lame-proof}
\eta v(y)
+
\mathcal A^\#\bigl(x_0,\xi+i\nu \partial_y\bigr)v(y)
=
0,
\quad y>0,
\end{equation}
together with the principal boundary condition
\begin{equation}
\label{eq:LS-lame-proof-bc}
\mathcal B^\#\bigl(x_0,\xi+i\nu \partial_y\bigr)v(0)=h.
\end{equation}
After a rigid rotation, we may assume \(\nu=e_3\) and \(\xi=(\xi_1,\xi_2,0)\). For the uniqueness part it suffices to consider the homogeneous boundary condition \(h=0\).
We introduce the Fourier-modified gradient
\[
\nabla_{\xi,y} v
:=
\begin{pmatrix}
i\xi_1 v\\
i\xi_2 v\\
\partial_y v
\end{pmatrix} \ \text{ and } \
\div_{\xi,y} v
:=
i\xi_1 v_1 + i\xi_2 v_2 + \partial_y v_3,
\]
and the corresponding symmetric gradient
\[
D_{\xi,y}(v)
:=
\frac12\Bigl(\nabla_{\xi,y}v + (\nabla_{\xi,y}v)^\top\Bigr).
\]
Then \eqref{eq:LS-lame-proof} may be rewritten as
\[
\eta v
-
\frac{1}{\varrho_0(x_0)}
\div_{\xi,y}\Bigl(
2\mu D_{\xi,y}(v)
+
\lambda \,\div_{\xi,y}(v)\,I_3
\Bigr)
=
0,
\quad y>0.
\]
Assume that \(v\in \rC_0(\R_+;\C^3)\) is a decaying solution of \eqref{eq:LS-lame-proof}--\eqref{eq:LS-lame-proof-bc} with \(h=0\). Multiplying by \(\overline{v}\), integrating over \((0,\infty)\), and using integration by parts together with the homogeneous traction condition at \(y=0\), we obtain
\[
\mathrm{Re}\, \eta \int_0^\infty |v(y)|^2\,\d y
+
\frac{1}{\varrho_0(x_0)}
\int_0^\infty
\Bigl(
2\mu |D_{\xi,y}(v)(y)|^2
+
\lambda |\div_{\xi,y}v(y)|^2
\Bigr)\,\d y
=
0.
\]
Since \(\mathrm{Re}\, \eta>0\), \(\mu>0\), and \(2\mu+3\lambda>0\), the quadratic form
\[
E\mapsto 2\mu |E|^2 + \lambda |\operatorname{tr}E|^2
\]
is positive definite on complex symmetric \(3\times3\)-matrices. Hence the second integral is nonnegative, and therefore both terms vanish. It follows that
\[
v\equiv 0.
\]
Thus the homogeneous half-line problem admits only the trivial decaying solution. 
Now, for fixed \((x_0,\xi,\eta)\), the space of decaying solutions of \eqref{eq:LS-lame-proof} has dimension \(3\), since we are dealing with a second-order \(3\times 3\)-system and the stable subspace has dimension equal to the number of boundary conditions. Therefore the boundary map
\[
v \longmapsto \mathcal B^\#\bigl(x_0,\xi+i\nu\partial_y\bigr)v(0)
\]
is a linear map from a \(3\)-dimensional space into \(\C^3\). Since its kernel is trivial by the uniqueness argument above, it is an isomorphism. Consequently, for every \(h\in \C^3\), there exists a unique decaying solution of \eqref{eq:LS-lame-proof}--\eqref{eq:LS-lame-proof-bc}. This proves the Lopatinskii--Shapiro condition.
\end{step}
\end{proof}
\noindent
We now turn to the linear stochastic problem \eqref{eq:U-system}. By the preceding analysis, there exists $\omega_A >0$ such that the realization $A_q+\omega_A $ admits a bounded $\Hinfty$-calculus on $\rLq(\Omega_0;\R^3)$ with $\Hinfty$-angle strictly less than $\pi/2$, see \cite[Theorem~2.3]{MR2047641}. Hence the stochastic maximal regularity result for operators with bounded $\Hinfty$-calculus applies to the shifted problem; see \cite{MR2952092,MR4459102}.
\noindent
Since \(A_q+\omega_A \) admits a bounded \(\Hinfty\)-calculus on \(\rLq(\Omega_0;\R^3)\) with \(\Hinfty\)-angle strictly less than \(\pi/2\), permanence properties of the \(\Hinfty\)-calculus imply that one may shift the operator along the fractional domain scale. More precisely, for \( {0 < \eps < 1/(2q)}\), we set
\[
\rX_0^\eps:=\rD\bigl((A_q+\omega_A )^\eps\bigr) \ \text{ and } \
\rX_1^\eps:=\rD\bigl((A_q+\omega_A )^{1+\eps}\bigr),
\]
and denote by \(A_q^\eps\) the realization of \(\mathcal A(x,D)\) in the ground space \(\rX_0^\eps\), i.e.
\[
A_q^\eps u:=\mathcal A(x,D)u,
\quad
\rD(A_q^\eps)=\rX_1^\eps.
\]
Then \(A_q^\eps+\omega_A \) again admits a bounded \(\Hinfty\)-calculus on \(\rX_0^\eps\), with \(\Hinfty\)-angle strictly less than \(\pi/2\).
Hence, the stochastic maximal regularity result applies to the shifted problem
\[
\d V + (A_q^\eps+\omega_A )V\,\d t
=
e^{-\omega_A t}f_{\mathrm{sto}}\,\d B_t,
\quad
V(0)=0.
\]
Setting
\[
U(t):=e^{\omega_A t} V(t),
\]
It\^o's product rule yields
\[
\d U
=
\omega_A  e^{\omega_A t} V\,\d t + e^{\omega_A t} \,\d V
=
-A_q^\eps U\,\d t + f_{\mathrm{sto}}\,\d B_t.
\]
Thus \(U\) solves the unshifted problem associated with the realization \(A_q^\eps\),
\[
\d U + A_q^\eps U\,\d t = f_{\mathrm{sto}}\,\d B_t,
\quad
U(0)=0.
\]
Since we work on the finite time interval \((0,T)\), the factors \(e^{\pm\omega_A t} \) are bounded on \([0,T]\). Therefore multiplication by \(e^{\pm\omega_A t} \) defines an isomorphism on the relevant time-dependent spaces, so that the estimates for \(V\) and \(U\) are equivalent up to constants depending on \(T\) and \(\omega_A \). Moreover, the domains of \(A_q^\eps\) and \(A_q^\eps+\omega_A \) coincide, and hence so do the associated interpolation spaces, up to equivalence of norms. For this reason, the shift by \(\omega_A \) is again inessential on finite time intervals. The introduction of the shifted realization \(A_q^\eps\) is crucial for the nonlinear analysis, since the stochastic convolution \(U\) enters the nonlinear terms, in particular the boundary term \(F_\Gamma\), and the regularity obtained from stochastic maximal regularity in the original ground space \(\rLq(\Omega_0;\R^3)\) is not sufficient for the corresponding estimates.

{ 
Now set 
$
\rX_\gamma^\eps = (\rX_0^\eps, \rX_1^\eps)_{1- \nicefrac{1}{p}, p}$  and  
    $\rX_\theta^\eps = [ \rX_0^\eps, \rX_1^\eps]_{\theta}$ where $(\cdot, \cdot)_{\theta,p}$ denotes the real and $[\cdot, \cdot]_{\theta}$ denotes the complex interpolation functor for $\theta \in (0,1)$ and $p \in (1,\infty)$. More precisely, 
for $0<\eps<1/(2q)$, we have
\[
\begin{aligned}
\rX_\gamma^\eps
= (\rX_0^\eps,\rX_1^\eps)_{1-\nicefrac{1}{p},p}
=
\begin{cases}
\rB^{2+2\eps-\nicefrac{2}{p}}_{qp}(\Omega_0;\R^3),
&
\displaystyle \nicefrac{2}{p}+\nicefrac{1}{q}>1+2\eps,
\\
\bigl\{
u\in\rB^{2+2\eps-\nicefrac{2}{p}}_{qp}(\Omega_0;\R^3)
\colon
\mathcal B(x,D)u=0 \text{ on }\Gamma_0
\bigr\},
&
\displaystyle \nicefrac{2}{p}+\nicefrac{1}{q}<1+2\eps,
\end{cases}
\end{aligned}
\]
and
\[
\begin{aligned}
\rX_\theta^\eps
= [\rX_0^\eps,\rX_1^\eps]_\theta
=
\begin{cases}
\rH^{2(\theta+\eps),q}(\Omega_0;\R^3),
&
\displaystyle \theta+\eps<\nicefrac12+\nicefrac{1}{2q},
\\
\bigl\{
u\in\rH^{2(\theta+\eps),q}(\Omega_0;\R^3)
\colon
\mathcal B(x,D)u=0 \text{ on }\Gamma_0
\bigr\},
&
\displaystyle \theta+\eps>\nicefrac12+\nicefrac{1}{2q}.
\end{cases}
\end{aligned}
\]}
Then the following result can be found in \cite{MR4459102}.
\begin{lem}\label{lem:stochastic-U-additive}
Let \(p\in(2,\infty)\), \(q\in(3,\infty)\), and let \( {0<\eps <1/(2q)} \).
Assume that 
\[
f_{\mathrm{sto}}
\in
\rL_\mathscr{P}^0\bigl(\Omega;\rLp(0,T;{ \gamma_{\mathrm{rad}}}(\mathcal{U},\rX^\eps_{1/2}))\bigr).
\]
Then the stochastic problem
\[
\left\{
\begin{aligned}
\d U + A_q^\eps U\,\d t &= f_{\mathrm{sto}}\,\d B_t,
&& \text{in } (0,T),
\\
U(0) &= 0,
\end{aligned}
\right.
\]
admits a unique strong adapted solution
\[
U(\omega)
\in \E_{1,T}^{\mathrm{sto}} :=
\bigcap_{\theta \in [0, \nicefrac{1}{2})}  \rH^{\theta,p}(0,T;\rX_{1-\theta}^\eps)
\cap
\rC([0,T];\rX_\gamma^\eps),
\]
{ 
and it holds that
\[
\|U(\omega)\|_{\E_{1,T}^{\mathrm{sto}}} = \|U(\omega)\|_{\rLp(0,T;\rX_1^\eps)}
+
\|U(\omega)\|_{\rH^{\theta,p}(0,T;\rX_{1-\theta}^\eps)}
+
\|U(\omega)\|_{\rC([0,T];\rX_\gamma^\eps)}
\le C_\theta(\omega),
\]
where \(C_\theta(\omega)<\infty\) almost surely is independent of \(T\),
and the norm is understood, for a fixed \(\theta\in(0,\frac12)\).
}
\end{lem}
\noindent
The preceding two lemmas allow us to eliminate the additive noise by a standard Da Prato--Debussche type decomposition. More precisely, let \(U\) denote the unique strong adapted solution of \eqref{eq:U-system} guaranteed by \autoref{lem:stochastic-U-additive}
and define
\[
\bar u := v+U.
\]
Then \(\bar u\) solves
\begin{equation}
\left\{
\begin{aligned}
\d \bar u - \varrho_0^{-1}\operatorname{div}\mathbb S(\nabla \bar u)\,\d t
&= f_u\,\d t + f_{\mathrm{sto}}\,\d B_t,
&& \text{in } \Omega_0\times(0,T),
\\
\mathbb S(\nabla \bar u)N &= g,
&& \text{on } \Gamma_0\times(0,T),
\\
\bar u(0)&=u_0,
&& \text{in } \Omega_0,
\end{aligned}
\right.
\label{eq:combined-u-additive}
\end{equation}
if and only if \(v\) solves the pathwise deterministic Lam\'e problem \eqref{eq:v-system}.
Consequently, the additive stochastic forcing is completely encoded in the fixed stochastic convolution \(U\), whereas the remaining unknown \(v\) satisfies a deterministic parabolic boundary value problem with random coefficients only through the dependence on \(U\).

\smallskip

\noindent
For the nonlinear problem, the stochastic convolution \(U\) is first constructed once and for all, and is then regarded as fixed. 
Using the notation introduced in \eqref{eq:transformed-FB-CNSE-lame-divrho-F2} and \eqref{eq:transformed-FB-CNSE-lame-divrho-FGamma}, the stochastic evolution problem is recast as a pathwise deterministic system with random coefficients,
\begin{equation}
\left\{
\begin{aligned}
\partial_t v + \mathcal A(x,D)v
&=  F_u(v+U),
\\
\mathcal B(x,D)v
&=  F_\Gamma(v+U),
\\
v(0)&=u_0.
\end{aligned}
\right.
\label{eq:recast-system-v}
\end{equation}
Here \(J\) and \(\mathrm Z\) are the Lagrangian quantities generated by the velocity field \(v+U\). In this formulation, the additive noise no longer appears explicitly in the evolution equations; it enters only through the fixed random field \(U\).

\section{Nonlinear Estimates}
\label{sec: estimates}
\noindent
This section is devoted to the nonlinear estimates for the terms $F_u$, and $F_\Gamma$ defined in \eqref{eq:transformed-FB-CNSE-lame-divrho-F2} and \eqref{eq:transformed-FB-CNSE-lame-divrho-FGamma}. 
We first record a paraproduct bound that will be used to control products in the nonlinear boundary condition $F_\Gamma$ in Triebel--Lizorkin norms. This type of estimate is a key ingredient in the proof of \autoref{thm:local-eulerian} and goes back essentially to Chae~\cite{Chae:02}. For related bilinear estimates in homogeneous Triebel--Lizorkin spaces in the context of the Navier--Stokes equations, we refer to \cite{KS:04}.

\begin{lem}[Paraproduct estimates in $\rF^s_{pq}$]
\label{lem:weighted paraprod estimate} \mbox{} \\
Let $s > 0$, $p \in (1,\infty)$, $q \in (1,\infty]$ and $p_1, p_2, r_1, r_2 \in [1,\infty]$ satisfying $\frac{1}{p} = \frac{1}{p_1} + \frac{1}{p_2} = \frac{1}{r_1} + \frac{1}{r_2}$.
Then there exists a constant $C > 0$ such that it holds that
\begin{equation*}
    \| f g \|_{\rF_{pq}^s(0,T)} \le C\bigl(\| f \|_{\rL^{p_1}(0,T)} \| g \|_{\rF_{p_2 q}^s(0,T)} + \| g \|_{\rL^{r_1}(0,T)} \| f \|_{\rF_{r_2 q}^s(0,T)}\bigr).
\end{equation*}
\end{lem}
We are now in a position to state the key estimates for the nonlinearities, which form the backbone of the fixed point argument.
\begin{lem}[Pathwise estimates for the nonlinearities]
\label{lem:pathwise-estimates-Fu-FGamma} \mbox{} \\
Let \(T>0\), \(p\in(2,\infty)\), \(q\in(3,\infty)\) satisfy
\(
\frac{2}{p}+\frac{3}{q}<1
\), \(0 < \eps < 1/(2q) \), \(\varrho_0(\omega) \in \rH^{1,q}(\Omega_0)\) and $f_{\mathrm{sto}}(\omega)
\in
\rLp(0,T;{ \gamma_{\mathrm{rad}}}(\mathcal{U},\rX^\eps_{1/2}))$. Fix \(R>0\) and \(\delta\in(0,\delta_0]\). Assume that
\[
v(\omega)\in \mathbb E_{1,T}, \quad v(0) =u_0,\quad \|v(\omega)\|_{\mathbb E_{1,T}} \leq R
\]
and set $ \bar u(\omega):=v(\omega)+U(\omega)$,
where \(U\) denotes the stochastic convolution constructed in \autoref{lem:stochastic-U-additive}.
Let \(\rX(\omega)\), \(\mathrm Z(\omega)\), and \(J(\omega)\) be the associated stochastic Lagrangian quantities, and define
\[
\bar\varrho(\omega):=\frac{\varrho_0(\omega)}{J(\omega)}, \quad g_0(\omega) := (p(\varrho_0(\omega)) - p_{\mathrm{ext})}N
\]
Moreover, for $\theta \in (0, \nicefrac{1}{2})$ assume that
\[
\theta=\frac12-\frac1{2q}.
\] 
Denote by 
\[ 
\tau_{\delta,R, \theta}>0
\]
the $(\mathcal{F}_t)_t$-stopping time defined in 
\autoref{prop:pathwise-stochastic-lagrangian}.
Then, on \([0,\tau_{\delta,R,\theta}(\omega)]\), there exists a constant \(C(\omega)>0\), depending only on \(p\), \(q\), \(\gamma\), \(\eps\), \(\Omega_0\), \(\mu\), \(\lambda\), \(\varrho_\ast\), \(p_{\mathrm{ext}}\), \(\delta_0\), \(M_{\varrho_0}(\omega)\), and \(M_{\varrho_0}^{-1}(\omega)\), and the fixed flow bounds \(\Lambda(T),K_\alpha(T)\), where $\alpha \in ( \theta, 1/2)$, such that
\begin{equation*}
\begin{aligned}
\|F_u(v(\omega)+U(\omega))\|_{\mathbb E_{0,\tau_{\delta,R,\theta}}}
&\le
C(\omega)\bigl(
\delta\,\bigl(R+ M_{u_0}(\omega)+M_{\mathrm{sto}}(\omega)\bigr)
+
\tau_{\delta,R,\theta}(\omega)^{1/p}
\bigr),
\\
\|F_\Gamma(v(\omega)+U(\omega)) -g_0 (\omega )\|_{\mathbb G_{\tau_{\delta,R,\theta}}}
&\le
C(\omega)\bigl(
\delta\,\bigl(1+R+ M_{u_0}(\omega) +M_{\mathrm{sto}}(\omega)\bigr)
+
\tau_{\delta,R,\theta}(\omega)^{1/p}
\bigr),
\end{aligned}
\end{equation*}
where
\[
\begin{gathered}
M_{u_0}(\omega) = \| u_0(\omega)\|_{\rX_\gamma}, \ 
M_{\varrho_0}(\omega)=\|\varrho_0(\omega)\|_{\rH^{1,q}(\Omega_0)},
\
M_{\varrho_0}^{-1}(\omega)=\Bigl\|\frac1{\varrho_0(\omega)}\Bigr\|_{\rH^{1,q}(\Omega_0)}, \
M_{\mathrm{sto}}(\omega)=
\|U(\omega)\|_{\E_{1,T}^{\mathrm{sto}}} .
\end{gathered}
\]

\end{lem}

\begin{proof}
Let \(\omega\in\Omega\) be fixed. To ease notation, we suppress the dependence on
\(\omega\) in what follows and write
\[
\tau:=\tau_{\delta,R,\theta}(\omega).
\]
From \autoref{prop:pathwise-stochastic-lagrangian} we recall
\begin{equation*}
\|\mathrm Z-I_3\|_{\rL^\infty(0,\tau;\rH^{1,q}(\Omega_0))}
+
\|J-1\|_{\rL^\infty(0,\tau;\rH^{1,q}(\Omega_0))}
\le C\,\delta,
\end{equation*}
and
\begin{equation*}
\|\mathrm Z-I_3\|_{\rH^{\theta,p}(0,\tau;\rH^{1,q}(\Omega_0))}
+
\|J-1\|_{\rH^{\theta,p}(0,\tau;\rH^{1,q}(\Omega_0))}
\le \delta.
\end{equation*}
In particular,
\begin{equation*}
\|\mathrm Z\|_{\rL^\infty(0,\tau;\rH^{1,q}(\Omega_0))}
+
\|J\|_{\rL^\infty(0,\tau;\rH^{1,q}(\Omega_0))}
\le C.
\end{equation*}
Moreover,
\begin{equation*}
\|J-1\|_{\rL^p(0,\tau;\rH^{1,q}(\Omega_0))}
\le
C\,\tau^{1/p}\delta,
\qquad
\|J\|_{\rL^p(0,\tau;\rH^{1,q}(\Omega_0))}
\le
C\,\tau^{1/p}.
\end{equation*}
Here and below we set
\[
M_{\varrho_0}:=\|\varrho_0\|_{\rH^{1,q}(\Omega_0)},
\qquad
M_{\varrho_0}^{-1}:=
\Bigl\|\frac1{\varrho_0}\Bigr\|_{\rH^{1,q}(\Omega_0)}.
\]
Since \(\rH^{1,q}(\Omega_0)\) is a Banach algebra, the bounds for \(J\) and
\(\mathrm Z\) imply
\begin{equation*}
\Bigl\|\frac{J-1}{\varrho_0}\Bigr\|_{\rL^\infty(0,\tau;\rH^{1,q}(\Omega_0))}
\le
C\,M_{\varrho_0}^{-1}\,\delta
\ \text{ and }\
\Bigl\|\frac{J}{\varrho_0}\Bigr\|_{\rL^\infty(0,\tau;\rH^{1,q}(\Omega_0))}
\le
C\,M_{\varrho_0}^{-1}.
\end{equation*}
Moreover,
\begin{equation*}
\Bigl\|\frac{J}{\varrho_0}\Bigr\|_{\rL^p(0,\tau;\rH^{1,q}(\Omega_0))}
\le
C\,\tau^{1/p}M_{\varrho_0}^{-1}.
\end{equation*}
Since \(\bar u=v+U\), \(\|v\|_{\mathbb E_{1,T}}\le R\), and
\(
\mathbb E_{1,T}\hookrightarrow \rL^p(0,T;\rH^{2,q}(\Omega_0))
\),
it follows from \autoref{lem:stochastic-U-additive} that
\begin{equation*}
\|\bar u\|_{\rL^p(0,\tau;\rH^{2,q}(\Omega_0))}
\le
\|v+U\|_{\rL^p(0,\tau;\rH^{2,q}(\Omega_0))}
\le
C\bigl(R+M_{u_0}+M_{\mathrm{sto}}\bigr),
\end{equation*}
where
\[
M_{\mathrm{sto}}
:=
\|U\|_{\E_{1,T}^{\mathrm{sto}}}.
\]
We now estimate \(F_u\). Recall that \(F_u\) is given by
\eqref{eq:transformed-FB-CNSE-lame-divrho-F2}. For the first term, we obtain
\begin{equation*}
\begin{aligned}
\Bigl\|
\frac{J-1}{\varrho_0}
\nabla^2\bar u
\Bigr\|_{\rL^p(0,\tau;\rL^q(\Omega_0))}
&\le
C
\Bigl\|\frac{J-1}{\varrho_0}\Bigr\|_{\rL^\infty(0,\tau;\rH^{1,q}(\Omega_0))}
\|\bar u\|_{\rL^p(0,\tau;\rH^{2,q}(\Omega_0))}
\\
&\le
C\,\delta\,M_{\varrho_0}^{-1}
\bigl(R+M_{u_0}+M_{\mathrm{sto}}\bigr).
\end{aligned}
\end{equation*}
For the terms involving second derivatives of \(\bar u\) and factors
\(\mathrm Z-I_3\), we use the boundedness of \(J/\varrho_0\), \(J\), and
\(\mathrm Z\) in \(\rL^\infty(0,\tau;\rH^{1,q}(\Omega_0))\). This gives
\begin{equation*}
\begin{aligned}
\Bigl\|
\frac{J}{\varrho_0}
(\mathrm Z-I_3)\nabla^2\bar u
\Bigr\|_{\rL^p(0,\tau;\rL^q(\Omega_0))}
&\le
C
\Bigl\|\frac{J}{\varrho_0}\Bigr\|_{\rL^\infty(0,\tau;\rH^{1,q})}
\|\mathrm Z-I_3\|_{\rL^\infty(0,\tau;\rH^{1,q})}
\|\bar u\|_{\rL^p(0,\tau;\rH^{2,q})}
\\
&\le
C\,\delta\,
\bigl(R+M_{u_0}+M_{\mathrm{sto}}\bigr).
\end{aligned}
\end{equation*}
The same estimate applies to all second-order terms in \(F_u\) which contain
one factor \(\mathrm Z-I_3\).
For the first-order terms containing \(\nabla\bar u\,\nabla\mathrm Z\), we use
\(\nabla \mathrm Z=\nabla(\mathrm Z-I_3)\). Hence
\begin{equation*}
\begin{aligned}
\Bigl\|
\frac{J}{\varrho_0}
\nabla \bar u\,\nabla \mathrm Z
\Bigr\|_{\rL^p(0,\tau;\rL^q(\Omega_0))}
&\le
C
\Bigl\|\frac{J}{\varrho_0}\Bigr\|_{\rL^\infty(0,\tau;\rH^{1,q})}
\|\nabla \bar u\|_{\rL^p(0,\tau;\rL^\infty)}
\|\nabla(\mathrm Z-I_3)\|_{\rL^\infty(0,\tau;\rL^q)}
\\
&\le
C
\|\bar u\|_{\rL^p(0,\tau;\rH^{2,q})}
\|\mathrm Z-I_3\|_{\rL^\infty(0,\tau;\rH^{1,q})}
\\
&\le
C\,\delta\,
\bigl(R+M_{u_0}+M_{\mathrm{sto}}\bigr).
\end{aligned}
\end{equation*}
Thus all non-pressure terms in \(F_u\) are bounded by
\[
C(\omega)\,\delta\,\bigl(R+M_{u_0}+M_{\mathrm{sto}}\bigr).
\]
It remains to estimate the pressure term. Since
\[
p\left(\frac{\varrho_0}{J}\right)
=
a\left(\frac{\varrho_0}{J}\right)^\gamma,
\]
the composition estimate in \(\rH^{1,q}(\Omega_0)\), the lower bound for \(J\),
and the \(\rL^\infty(0,\tau;\rH^{1,q})\)-bounds for \(J\) imply
\begin{equation*}
\left\|
p\left(\frac{\varrho_0}{J}\right)
\right\|_{\rL^\infty(0,\tau;\rH^{1,q}(\Omega_0))}
\le
C(\omega).
\end{equation*}
Consequently,
\begin{equation*}
\left\|
\nabla p\left(\frac{\varrho_0}{J}\right)
\right\|_{\rL^p(0,\tau;\rL^q(\Omega_0))}
\le
C(\omega)\,\tau^{1/p}.
\end{equation*}
Using also the \(\rL^\infty(0,\tau;\rH^{1,q})\)-bounds for
\(\frac{J}{\varrho_0}\) and \(\mathrm Z\), we obtain
\begin{equation*}
\begin{aligned}
\Bigl\|
\frac{J}{\varrho_0}\mathrm Z^\top
\nabla p\left(\frac{\varrho_0}{J}\right)
\Bigr\|_{\rL^p(0,\tau;\rL^q(\Omega_0))}
&\le
C
\Bigl\|\frac{J}{\varrho_0}\Bigr\|_{\rL^\infty(0,\tau;\rH^{1,q})}
\|\mathrm Z\|_{\rL^\infty(0,\tau;\rH^{1,q})}
\left\|
\nabla p\left(\frac{\varrho_0}{J}\right)
\right\|_{\rL^p(0,\tau;\rL^q)}
\\
&\le
C(\omega)\,\tau^{1/p}.
\end{aligned}
\end{equation*}
Collecting the above bounds, we arrive at
\begin{equation*}
\|F_u(v+U)\|_{\mathbb E_{0,\tau}}
\le
C(\omega)\bigl(
\delta\,\bigl(R+M_{u_0}+M_{\mathrm{sto}}\bigr)
+
\tau^{1/p}
\bigr).
\end{equation*}
Finally, to estimate \(F_\Gamma\), we collect some useful bounds for the
geometric boundary coefficients. Let
\(\widetilde N\in \rC^1(\overline{\Omega}_0;\R^3)\) be a fixed extension of
the outer unit normal \(N\) from \(\Gamma_0\) to \(\Omega_0\), which exists
since \(\Gamma_0\) is of class \(\rC^2\). Then
\[
\widetilde N|_{\Gamma_0}=N,
\qquad
\widetilde N-J\mathrm Z^\top \widetilde N=(I_3-J\mathrm Z^\top)\widetilde N,
\qquad
I_3-J\mathrm Z^\top=(1-J)I_3+J(I_3-\mathrm Z^\top).
\]
Since \(\rH^{1,q}(\Omega_0)\) is a Banach algebra for \(q>3\), it follows from
the bounds for \(J-1\) and \(\mathrm Z-I_3\) that
\begin{equation*}
\begin{aligned}
\|\widetilde N-J\mathrm Z^\top \widetilde N\|_{\rL^\infty(0,\tau;\rH^{1,q}(\Omega_0))}
&\le
C\bigl(
\|J-1\|_{\rL^\infty(0,\tau;\rH^{1,q}(\Omega_0))}
+
\|\mathrm Z-I_3\|_{\rL^\infty(0,\tau;\rH^{1,q}(\Omega_0))}
\bigr)
\\
&\le C\delta,
\end{aligned}
\end{equation*}
and
\begin{equation*}
\|J\mathrm Z^\top \widetilde N\|_{\rL^\infty(0,\tau;\rH^{1,q}(\Omega_0))}
\le C.
\end{equation*}
By the trace theorem and the embedding
\[
\rH^{1,q}(\Omega_0)\hookrightarrow \rC(\overline{\Omega}_0)\hookrightarrow \rL^\infty(\Gamma_0),
\]
we infer
\begin{equation*}
\|N-J\mathrm Z^\top N\|_{\rL^\infty(0,\tau;\rL^\infty(\Gamma_0))}
\le C\delta,
\qquad
\|J\mathrm Z^\top N\|_{\rL^\infty(0,\tau;\rL^q(\Gamma_0))}
\le C.
\end{equation*}
Next, recall that
\[
s=\frac12-\frac1{2q}
\]
and that \(\theta\ge s\). Using again that
\(\rH^{1,q}(\Omega_0)\) is a Banach algebra and that \(\widetilde N\) is
time-independent, together with the paraproduct estimate \autoref{lem:weighted paraprod estimate}, we get
\begin{equation*}
\begin{aligned}
\|N-J\mathrm Z^\top N\|_{\rF_{pq}^{s}(0,\tau;\rL^\infty(\Gamma_0))}
&\le
C\|\widetilde N-J\mathrm Z^\top \widetilde N\|_{\rF_{pq}^{s}(0,\tau;\rH^{1,q}(\Omega_0))}
\\
&\le
C\bigl(
\|J-1\|_{\rH^{s,p}(0,\tau;\rH^{1,q}(\Omega_0))}
+
\|\mathrm Z-I_3\|_{\rH^{s,p}(0,\tau;\rH^{1,q}(\Omega_0))}
\bigr)
\\
&\le C\delta.
\end{aligned}
\end{equation*}
Moreover, writing
\[
J\mathrm Z^\top \widetilde N
=
\widetilde N + (J\mathrm Z^\top-I_3)\widetilde N
\ \text{ and }\
J\mathrm Z^\top-I_3=(J-1)\mathrm Z^\top+(\mathrm Z^\top-I_3),
\]
we obtain
\begin{equation*}
\begin{aligned}
\|J\mathrm Z^\top N\|_{\rF_{pq}^s(0,\tau;\rL^q(\Gamma_0))}
&\le
C\|\widetilde N\|_{\rH^{s,p}(0,\tau;\rH^{1,q}(\Omega_0))}
+
C\|(J\mathrm Z^\top-I_3)\widetilde N\|_{\rF_{pq}^{s}(0,\tau;\rH^{1,q}(\Omega_0))}
\\
&\le
C
+
C\bigl(
\|J-1\|_{\rH^{s,p}(0,\tau;\rH^{1,q}(\Omega_0))}
+
\|\mathrm Z-I_3\|_{\rH^{s,p}(0,\tau;\rH^{1,q}(\Omega_0))}
\bigr)
\\
&\le
C\bigl(1+\delta\bigr),
\end{aligned}
\end{equation*}
using the trace theorem and { the embedding}
{ \[
\rH^{s,p}(0,\tau;\rL^q(\Gamma_0))
\hookrightarrow
\rF_{pq}^s(0,\tau;\rL^q(\Gamma_0)),
\qquad q>2.
\]}
{ This embedding also holds with values in \(\rH^{1,q}(\Omega_0)\)
and the fractional operator domains, since these spaces are UMD spaces of
cotype \(q\). For the \(\rL^\infty(\Gamma_0)\)-valued coefficient estimate,
we apply it in \(\rH^{1,q}(\Omega_0)\) before taking the spatial trace.}
These estimates will be used repeatedly in the bounds for the different terms in
\(F_\Gamma\).
Since
\[
\mathbb G_{\tau}
=
\rF_{pq}^{\nicefrac12-\nicefrac{1}{2q}}(0,\tau;\rL^q(\Gamma_0;\R^3))
\cap
\rL^p(0,\tau;\rB_{qq}^{1-\nicefrac1q}(\Gamma_0;\R^3)),
\]
it suffices to estimate separately the two norms. We first consider the
\(\rL^p(0,\tau;\rB_{qq}^{1-\nicefrac1q}(\Gamma_0;\R^3))\)-part.
By the trace theorem and the bounds for the geometric coefficient
\(N-J\mathrm Z^\top N\) established above, we obtain
\begin{equation*}
\begin{aligned}
\Big\|
\frac{\partial \bar u_i}{\partial y_j}
\Bigl(
N_j-\sum_l J\,\mathrm Z_{l,j}N_l
\Bigr)
\Big\|_{\rL^p(0,\tau;\rB_{qq}^{1-\nicefrac1q}(\Gamma_0))}
&\le
C
\Big\|
\frac{\partial \bar u_i}{\partial y_j}
\Bigl(
\widetilde N_j-\sum_l J\,\mathrm Z_{l,j}\widetilde N_l
\Bigr)
\Big\|_{\rL^p(0,\tau;\rH^{1,q}(\Omega_0))}
\\
&\le
C
\|\widetilde N-J\mathrm Z^\top \widetilde N\|_{\rL^\infty(0,\tau;\rH^{1,q}(\Omega_0))}
\|\bar u\|_{\rL^p(0,\tau;\rH^{2,q}(\Omega_0))}
\\
&\le
C\,\delta\,\bigl(R+M_{u_0}+M_{\mathrm{sto}}\bigr).
\end{aligned}
\end{equation*}
All remaining terms in \(F_\Gamma\) which contain first derivatives of
\(\bar u\) multiplied by either \(N-J\mathrm Z^\top N\) or
\(\mathrm Z-I_3\) are treated in the same way and satisfy the same bound.
It remains to treat the pressure boundary term. By the trace theorem,
\[
\left\|
p\left(\frac{\varrho_0}{J}\right)J\mathrm Z^\top N
\right\|_{\rL^p(0,\tau;\rB_{qq}^{1-\nicefrac1q}(\Gamma_0))}
\le
C
\left\|
p\left(\frac{\varrho_0}{J}\right)J\mathrm Z^\top \widetilde N
\right\|_{\rL^p(0,\tau;\rH^{1,q}(\Omega_0))}.
\]
Using the Banach algebra property of \(\rH^{1,q}(\Omega_0)\), the bounds for
\(J\mathrm Z^\top\widetilde N\), the lower bound for \(J\), and the composition
estimate in \(\rH^{1,q}(\Omega_0)\), we get
\begin{equation*}
\begin{aligned}
\left\|
p\left(\frac{\varrho_0}{J}\right)J\mathrm Z^\top \widetilde N
\right\|_{\rL^p(0,\tau;\rH^{1,q}(\Omega_0))}
&\le
C
\left\|
p\left(\frac{\varrho_0}{J}\right)
\right\|_{\rL^p(0,\tau;\rH^{1,q}(\Omega_0))}
\\
&\le
C(\omega)\,\tau^{1/p}.
\end{aligned}
\end{equation*}
Similarly, since \(p_{\mathrm{ext}}\) is constant,
\begin{equation*}
\begin{aligned}
\|p_{\mathrm{ext}}\,J\mathrm Z^\top N\|_{\rL^p(0,\tau;\rB_{qq}^{1-\nicefrac1q}(\Gamma_0))}
&\le
C\,
\|J\mathrm Z^\top N\|_{\rL^p(0,\tau;\rB_{qq}^{1-\nicefrac1q}(\Gamma_0))}
\\
&\le
C\,\tau^{1/p}.
\end{aligned}
\end{equation*}
Since \(\|g_0\|_{\rL^p(0,\tau;\rB_{qq}^{1-\nicefrac1q}(\Gamma_0;\R^3))}
\le C(\omega)\tau^{1/p}\), collecting the above estimates, we conclude that
\begin{equation*}
\begin{aligned}
\|F_\Gamma(v+U)-g_0\|_{\rL^p(0,\tau;\rB_{qq}^{1-\nicefrac1q}(\Gamma_0;\R^3))}
&\le
C\,\delta\,\bigl(R+M_{u_0}+M_{\mathrm{sto}}\bigr)
+
C(\omega)\,\tau^{1/p}.
\end{aligned}
\end{equation*}
Next, we estimate the
\[
\rF_{pq}^{s}(0,\tau;\rL^q(\Gamma_0;\R^3))
\]
part of \(\mathbb G_{\tau}\). We first consider the terms in
\(F_\Gamma\) which contain first derivatives of \(\bar u\) multiplied by the
geometric defect \(N-J\mathrm Z^\top N\). By the paraproduct estimate
\autoref{lem:weighted paraprod estimate}, together with the boundary coefficient
estimates established above, we obtain
\begin{equation*}
\begin{aligned}
&\Big\|
\frac{\partial \bar u_i}{\partial y_j}
\Bigl(
N_j-\sum_l J\,\mathrm Z_{l,j}N_l
\Bigr)
\Big\|_{\rF_{pq}^{s}(0,\tau;\rL^q(\Gamma_0))}
\\
&\le
C
\Big\|
N_j-\sum_l J\,\mathrm Z_{l,j}N_l
\Big\|_{\rL^\infty(0,\tau;\rL^\infty(\Gamma_0))}
\Big\|
\frac{\partial \bar u_i}{\partial y_j}
\Big\|_{\rF_{pq}^{s}(0,\tau;\rL^q(\Gamma_0))}
\\
&\quad
+
C
\Big\|
N_j-\sum_l J\,\mathrm Z_{l,j}N_l
\Big\|_{\rF_{pq}^{s}(0,\tau;\rL^\infty(\Gamma_0))}
\Big\|
\frac{\partial \bar u_i}{\partial y_j}
\Big\|_{\rL^\infty(0,\tau;\rL^q(\Gamma_0))}.
\end{aligned}
\end{equation*}
By the estimates for \(N-J\mathrm Z^\top N\) proved above, this yields
\begin{equation*}
\begin{aligned}
&\Big\|
\frac{\partial \bar u_i}{\partial y_j}
\Bigl(
N_j-\sum_l J\,\mathrm Z_{l,j}N_l
\Bigr)
\Big\|_{\rF_{pq}^{s}(0,\tau;\rL^q(\Gamma_0))}
\\
&\le
C\,\delta
\Biggl(
\Big\|
\frac{\partial \bar u_i}{\partial y_j}
\Big\|_{\rF_{pq}^{s}(0,\tau;\rL^q(\Gamma_0))}
+
\Big\|
\frac{\partial \bar u_i}{\partial y_j}
\Big\|_{\rL^\infty(0,\tau;\rL^q(\Gamma_0))}
\Biggr).
\end{aligned}
\end{equation*}
To estimate the trace terms involving \(\nabla\bar u\), we write
\[
\bar u=v+U.
\]
Since
\[
v\in \mathbb E_{1,T}
=
\rH^{1,p}(0,T;\rL^q(\Omega_0;\R^3))
\cap
\rL^p(0,T;\rH^{2,q}(\Omega_0;\R^3)),
\]
the anisotropic trace theorem yields
\[
\|\nabla v\|_{\rF_{pq}^{s}(0,\tau;\rL^q(\Gamma_0;\R^3))}
\le
C\bigl(\|v\|_{\mathbb E_{1,\tau}}+M_{u_0}\bigr)
\le
C\bigl(R+M_{u_0}\bigr).
\]
Concerning the stochastic convolution \(U\), we use that, by
\autoref{lem:stochastic-U-additive},
\[
U\in \rH^{s,p}(0,\tau;\rX^\eps_{1-s}),
\qquad s<\frac12,
\]
and
\[
\|U\|_{\rH^{s,p}(0,\tau;\rX^\eps_{1-s})}
\le M_{\mathrm{sto}}.
\]
Moreover, { by the embedding}
{ \[
\rH^{s,p}(0,\tau;\rX^\eps_{1-s})
\hookrightarrow
\rF_{pq}^{s}(0,\tau;\rX^\eps_{1-s}),
\qquad q>2,
\]}
and, by the embedding of the operator scale,
\[
\rX^\eps_{1-s}\hookrightarrow \rH^{2(1+\eps-s),q}(\Omega_0;\R^3)
=
\rH^{1+\nicefrac1q+2\eps,q}(\Omega_0;\R^3).
\]
Hence
\[
U\in
\rF_{pq}^{s}(0,\tau;\rH^{1+\nicefrac1q+2\eps,q}(\Omega_0;\R^3)),
\]
with
\[
\|U\|_{\rF_{pq}^{s}(0,\tau;\rH^{1+\nicefrac1q+2\eps,q}(\Omega_0;\R^3))}
\le
C\,M_{\mathrm{sto}}.
\]
Now, for each \(j\in\{1,2,3\}\), the trace operator
\[
\mathcal T_j :
\rH^{1+\nicefrac1q+2\eps,q}(\Omega_0;\R^3)\to \rL^q(\Gamma_0;\R^3),
\qquad
\mathcal T_j u := (\partial_j u)|_{\Gamma_0},
\]
is bounded. Therefore,
\[
\|\partial_j U\|_{\rF_{pq}^{s}(0,\tau;\rL^q(\Gamma_0;\R^3))}
\le
C\,M_{\mathrm{sto}}.
\]
Combining the previous estimates, we conclude that
\[
\|\partial_j \bar u\|_{\rF_{pq}^{s}(0,\tau;\rL^q(\Gamma_0;\R^3))}
\le
C\bigl(R+M_{u_0}+M_{\mathrm{sto}}\bigr).
\]
We next estimate the
\(
\rL^\infty(0,\tau;\rL^q(\Gamma_0;\R^3))
\)
norm. Since
\[
v\in \mathbb E_{1,T}
\hookrightarrow
\rC([0,\tau];\rB_{qp}^{2-\nicefrac2p}(\Omega_0;\R^3)),
\]
it follows that
\[
\nabla v\in
\rC([0,\tau];\rB_{qp}^{1-\nicefrac2p}(\Omega_0;\R^3)),
\]
and
\[
\|\nabla v\|_{\rL^\infty(0,\tau;\rB_{qp}^{1-\nicefrac2p}(\Omega_0;\R^3))}
\le
C\bigl(\|v\|_{\mathbb E_{1,\tau}}+M_{u_0}\bigr)
\le
C\bigl(R+M_{u_0}\bigr).
\]
Likewise, by the pathwise continuity of \(U\) from
\autoref{lem:stochastic-U-additive}, after passing to spatial Besov spaces in
the operator scale, we also have
\[
\nabla U\in
\rC([0,\tau];\rB_{qp}^{1-\nicefrac2p}(\Omega_0;\R^3)),
\]
and
\[
\|\nabla U\|_{\rL^\infty(0,\tau;\rB_{qp}^{1-\nicefrac2p}(\Omega_0;\R^3))}
\le
C\,M_{\mathrm{sto}}.
\]
Therefore
\[
\nabla \bar u\in
\rC([0,\tau];\rB_{qp}^{1-\nicefrac2p}(\Omega_0;\R^3)),
\]
with
\[
\|\nabla \bar u\|_{\rL^\infty(0,\tau;\rB_{qp}^{1-\nicefrac2p}(\Omega_0;\R^3))}
\le
C\bigl(R+M_{u_0}+M_{\mathrm{sto}}\bigr).
\]
Applying the trace theorem gives
\[
\nabla \bar u|_{\Gamma_0}
\in
\rL^\infty(0,\tau;\rB_{qp}^{1-\nicefrac2p-\nicefrac1q}(\Gamma_0;\R^3)).
\]
Since
\[
1-\frac2p-\frac1q>0,
\]
which follows from
\[
\frac2p+\frac3q<1,
\]
the Besov embedding on \(\Gamma_0\) yields
\[
\rB_{qp}^{1-\nicefrac2p-\nicefrac1q}(\Gamma_0;\R^3)
\hookrightarrow
\rL^q(\Gamma_0;\R^3).
\]
Consequently,
\[
\|\partial_j \bar u\|_{\rL^\infty(0,\tau;\rL^q(\Gamma_0;\R^3))}
\le
C\bigl(R+M_{u_0}+M_{\mathrm{sto}}\bigr).
\]
Altogether, we arrive at
\begin{equation}
\label{eq:pathwise-trace-est}
\Big\|
\frac{\partial \bar u_i}{\partial y_j}
\Big\|_{\rF_{pq}^{s}(0,\tau;\rL^q(\Gamma_0))}
+
\Big\|
\frac{\partial \bar u_i}{\partial y_j}
\Big\|_{\rL^\infty(0,\tau;\rL^q(\Gamma_0))}
\le
C\bigl(R+M_{u_0}+M_{\mathrm{sto}}\bigr).
\end{equation}
All remaining terms in \(F_\Gamma\) which contain first derivatives of
\(\bar u\) multiplied by either \(N-J\mathrm Z^\top N\) or
\(\mathrm Z-I_3\) are estimated in the same way and satisfy the same bound.
It remains to treat the pressure contribution in the
\[
\rF_{pq}^{s}(0,\tau;\rL^q(\Gamma_0;\R^3))
\]
norm. We write
\[
\begin{aligned}
\biggl(p\left(\frac{\varrho_0}{J}\right)-p_{\mathrm{ext}}\biggr)
J\mathrm Z^\top N-g_0
&=
\left[p\left(\frac{\varrho_0}{J}\right)-p(\varrho_0)\right]J\mathrm Z^\top N
\\
&\quad +(p(\varrho_0)-p_{\mathrm{ext}})(J\mathrm Z^\top N-N),
\qquad g_0:=(p(\varrho_0)-p_{\mathrm{ext}})N.
\end{aligned}
\]
By the trace theorem, \autoref{lem:aux-time-inverse-det}, and
the embedding
\[
\rH^{\theta,p}(0,\tau;\rH^{1,q}(\Omega_0))
\hookrightarrow
\rH^{s,p}(0,\tau;\rH^{1,q}(\Omega_0)),
\]
we obtain
\begin{equation*}
\begin{aligned}
\left\|
p\left(\frac{\varrho_0}{J}\right)-p(\varrho_0)
\right\|_{\rF_{pq}^{s}(0,\tau;\rL^\infty(\Gamma_0))}
&\le
C
\left\|
p\left(\frac{\varrho_0}{J}\right)-p(\varrho_0)
\right\|_{\rH^{s,p}(0,\tau;\rH^{1,q}(\Omega_0))}
\\
&\le
C(\omega)\,\delta.
\end{aligned}
\end{equation*}
Moreover,
\[
\left\|
p\left(\frac{\varrho_0}{J}\right)-p(\varrho_0)
\right\|_{\rL^\infty(0,\tau;\rL^\infty(\Gamma_0))}
\le
C(\omega)\,\delta.
\]
Hence, by the paraproduct estimate,
\begin{equation*}
\begin{aligned}
&\quad
\left\|
\left[p\left(\frac{\varrho_0}{J}\right)-p(\varrho_0)\right]
J\mathrm Z^\top N
\right\|_{\rF_{pq}^{s}(0,\tau;\rL^q(\Gamma_0))}
\\
&\le
C
\|J\mathrm Z^\top N\|_{\rL^\infty(0,\tau;\rL^q(\Gamma_0))}
\left\|
p\left(\frac{\varrho_0}{J}\right)-p(\varrho_0)
\right\|_{\rF_{pq}^{s}(0,\tau;\rL^\infty(\Gamma_0))}
\\
&\quad
+
C
\|J\mathrm Z^\top N\|_{\rF_{pq}^{s}(0,\tau;\rL^q(\Gamma_0))}
\left\|
p\left(\frac{\varrho_0}{J}\right)-p(\varrho_0)
\right\|_{\rL^\infty(0,\tau;\rL^\infty(\Gamma_0))}
\\
&\le
C(\omega)\,\delta.
\end{aligned}
\end{equation*}
Similarly, since \(p(\varrho_0)-p_{\mathrm{ext}}\) is time-independent and
bounded on \(\Gamma_0\),
\[
\|(p(\varrho_0)-p_{\mathrm{ext}})(J\mathrm Z^\top N-N)\|_{\rF_{pq}^{s}(0,\tau;\rL^q(\Gamma_0))}
\le
C(\omega)\,\delta.
\]
Collecting the above estimates, we conclude that
\begin{equation*}
\begin{aligned}
\|F_\Gamma(v+U)-g_0\|_{\rF_{pq}^{s}(0,\tau;\rL^q(\Gamma_0;\R^3))}
&\le
C\,\delta\,\bigl(R+M_{u_0}+M_{\mathrm{sto}}\bigr)
+
C(\omega)\bigl(\tau^{1/p}+\delta\bigr).
\end{aligned}
\end{equation*}
This completes the proof.
\end{proof}

\noindent
\noindent
We conclude this section by deriving pathwise difference estimates for the nonlinear terms \(F_u\) and \(F_\Gamma\) associated with two velocity fields. These estimates will be used later to establish the contraction property in the fixed-point argument for the transformed system.

\begin{cor}[Pathwise difference estimates for the nonlinearities]
\label{cor:pathwise-difference-nonlinearities}\mbox{} \\
Under the same assumptions and notation as in
\autoref{lem:pathwise-estimates-Fu-FGamma}, let
\[
v_1(\omega),\,v_2(\omega)\in \mathbb E_{1,T}
\ \text{ with }\
\|v_1(\omega)\|_{\mathbb E_{1,T}},\,
\|v_2(\omega)\|_{\mathbb E_{1,T}}\le R, \quad v_1(0) = v_2(0) = u_0.
\]
Set
\[
\bar u_i(\omega):=v_i(\omega)+U(\omega),
\qquad i=1,2,
\]
where \(U\) denotes the stochastic convolution from
\autoref{lem:stochastic-U-additive}. Let
\(
\rX_i(\omega),\, \mathrm Z_i(\omega),\, J_i(\omega),
\ i=1,2,
\)
be the associated stochastic Lagrangian quantities from
\autoref{prop:pathwise-stochastic-lagrangian}, and define
\[
\bar\varrho_i(\omega):=\frac{\varrho_0(\omega)}{J_i(\omega)},
\qquad i=1,2.
\]
Assume in addition that
\(
\theta=s =\frac12-\frac1{2q},
\)
and let
\(
\alpha\in(\theta,\nicefrac12).
\)
Then, on \([0,\tau_{\delta,R,\theta}(\omega)]\), there exists a constant
\(C(\omega)>0\), depending only on \(p\), \(q\), \(\gamma\), \(\eps\),
\(\Omega_0\), \(\mu\), \(\lambda\), \(\varrho_\ast\), \(p_{\mathrm{ext}}\),
\(\delta_0\), \(M_{\varrho_0}(\omega)\), and
\(M_{\varrho_0}^{-1}(\omega)\),and the fixed flow bounds \(\Lambda(T),K_\alpha(T)\), such that
\begin{equation*}
\begin{aligned}
&\quad
\|F_u(v_1(\omega)+U(\omega))
-
F_u(v_2(\omega)+U(\omega))\|_{\mathbb E_{0,\tau_{\delta,R,\theta}(\omega)}}
\\
&\le
C(\omega)\bigl(1+R+M_{u_0}+M_{\mathrm{sto}}(\omega)\bigr)
\bigl(\delta+\tau_{\delta,R,\theta}(\omega)^{1-\nicefrac1p}\bigr)
\|v_1-v_2\|_{\mathbb E_{1,\tau_{\delta,R,\theta}(\omega)}},
\\[0.5em]
&\quad
\|F_\Gamma(v_1(\omega)+U(\omega))
-
F_\Gamma(v_2(\omega)+U(\omega))\|_{\mathbb G_{\tau_{\delta,R,\theta}(\omega)}}
\\
&\le
C(\omega)\bigl(1+R+M_{u_0}+M_{\mathrm{sto}}(\omega)\bigr)
\bigl(
\delta
+
\tau_{\delta,R,\theta}(\omega)^{\alpha-s}
+
\tau_{\delta,R,\theta}(\omega)^{1-s}
\bigr)
\|v_1-v_2\|_{\mathbb E_{1,\tau_{\delta,R,\theta}(\omega)}}.
\end{aligned}
\end{equation*}
\end{cor}
\section{Local Well-posedness}\label{sec: local}
\noindent
\begin{proof}[Proof of \autoref{thm:local-eulerian}]\noindent
The argument is based on a pathwise fixed-point construction combined with a localization procedure. We first fix \(\omega\in\Omega\) and solve the transformed problem on a random time interval by a contraction argument, treating all stochastic quantities as frozen coefficients. We then localize in \(\omega\) so that the random constants appearing in the nonlinear estimates are bounded by deterministic quantities. This is achieved by combining an \(\mathcal F_0\)-measurable localization for the initial data with a stopping-time localization for the time-dependent stochastic terms. Finally, we consider the associated stopped fixed point iteration and show that the iterates are progressively measurable. Since the contraction is pathwise, the iterates converge to the unique pathwise fixed point, which is therefore progressively measurable and defines the desired local strong solution.
\setcounter{stp}{0}
\begin{step}[Pathwise fixed-point]\label{step 1 proof}\mbox{}\\
Let \(\omega\in\Omega\) be fixed in the following. For notational convenience, we suppress the dependence on \(\omega\) in what follows; in particular, all objects below are understood pathwise for this fixed sample \(\omega\). Suppose that $\varrho_0 \in \rH^{1,q}(\Omega_0)$ and $u_0 \in \rX_\gamma$.
Fix \(T_0\in(0,1]\) within the given time horizon, \(0<T\le T_0\), \(r>0\) , and \(R=2r\) , $\delta \in ( 0, \delta_0]$, and  
\[
\theta = \frac12-\frac1{2q}.
\] 
Define $\tau$ by   
\begin{equation*}
    \tau := \tau_{\delta, R, \theta} \wedge T,
\end{equation*}
where $\tau_{\delta, R, \theta}$ is given as in \autoref{prop:pathwise-stochastic-lagrangian}.
Then before fixing $\omega$, $\tau$ is a $(\mathcal{F}_t)_t$ adapted stopping time.
In order to center the iteration at zero and to incorporate the nonhomogeneous boundary condition at time \(t=0\), we introduce a reference solution \(v_{\mathrm{ref}}\) as the unique strong solution of
\begin{equation}
\left\{
\begin{aligned}
\partial_t v_{\mathrm{ref}} + \mathcal A(x,D)v_{\mathrm{ref}} &= 0,
&& \text{in } \Omega_0\times(0,T_0 ),
\\
\mathcal B(x,D)v_{\mathrm{ref}} &=
\bigl(p(\varrho_0)-p_{\mathrm{ext}}\bigr)N,
&& \text{on } \Gamma_0\times(0,T_0),
\\
v_{\mathrm{ref}}(0) &= u_0,
&& \text{in } \Omega_0.
\end{aligned}
\right.
\label{eq:vref-system}
\end{equation} 
By \autoref{lem:deterministic-v}, together with the compatibility condition
\[
\mathcal B(x,D)u_0
=
\bigl(p(\varrho_0)-p_{\mathrm{ext}}\bigr)N
\ \text{ on }\Gamma_0,
\]
the system \eqref{eq:vref-system} admits a unique solution
\(
v_{\mathrm{ref}}\in \mathbb E_{1,T_0}
.
\) 
As the norm of \(\mathbb E_{1,\tau}\) is given by time-integral norms, it follows that
\[
\|v_{\mathrm{ref}}\|_{\mathbb E_{1,\tau}}
\to 0
\ \text{ as } \ T\to  0.
\]
Next, define the solution ball $\mathbb B_r(0,\tau)$ by
\[
\mathbb B_r(0,\tau)
:=
\bigl\{
v \in \mathbb E_{1,\tau}:
v(0)=u_0,
\ \|v-v_{\mathrm{ref}}\|_{\mathbb E_{1,\tau}}\le r
\bigr\}.
\]
Then for $r>0$ and $T>0$ chosen sufficiently small and every $v\in \mathbb{B}_r(0,\tau)$, we obtain
\[
\|v\|_{\mathbb E_{1,\tau}}
\le
\|v-v_{\mathrm{ref}}\|_{\mathbb E_{1,\tau}}
+
\|v_{\mathrm{ref}}\|_{\mathbb E_{1,\tau}}
\le r+\|v_{\mathrm{ref}}\|_{\mathbb E_{1,\tau}}
\leq R.
\]
We now define the solution map
\[
\Psi:\mathbb B_r(0,\tau)\to  \mathbb E_{1,\tau},
\qquad
\Psi(v_1):=v,
\]
as follows. Let \(v_1\in \mathbb B_r(0,\tau)\) and set
\[
\bar u_1:=v_1+U,
\]
where \(U\) denotes the stochastic convolution solving \eqref{eq:U-system}. Let
\(
\rX_1,\, \mathrm Z_1=(\nabla\rX_1)^{-1},
\, J_1=\det\nabla\rX_1
\)
be the associated stochastic Lagrangian quantities generated by \(\bar u_1\). We recover the corresponding transformed density by
\[
\bar\varrho_1:=\frac{\varrho_0}{J_1}.
\]
Then \(v=\Psi(v_1)\) is defined as the unique solution of
\begin{equation}
\left\{
\begin{aligned}
\partial_t v+\mathcal A(x,D)v
&=
F_u(v_1+U),
&& \text{in } \Omega_0\times(0,\tau),
\\
\mathcal B(x,D)v
&=
F_\Gamma(v_1+U),
&& \text{on } \Gamma_0\times(0,\tau),
\\
v(0)&=u_0,
&& \text{in } \Omega_0.
\end{aligned}
\right.
\label{eq:fixed-point-v}
\end{equation}
Here \(F_u(v_1+U)\) and \(F_\Gamma(v_1+U)\) are evaluated with the Lagrangian quantities \(\mathrm Z_1\), \(J_1\), and the density \(\bar\varrho_1=\varrho_0/J_1\).
To see that the map \(\Psi\) is well defined, fix \(v_1\in\mathbb B_r(0,\tau)\). By construction,
\[
\|v_1\|_{\mathbb E_{1,\tau}}\le R.
\]
Hence, by \autoref{lem:pathwise-estimates-Fu-FGamma}, the corresponding nonlinear terms satisfy
\[
F_u(v_1+U)\in \mathbb E_{0,\tau}
\ \text{ and }\
F_\Gamma(v_1+U)\in \mathbb G_{\tau}.
\]
Thus \autoref{lem:deterministic-v} yields a unique solution \(v\in \mathbb E_{1,\tau}\), and therefore \(\Psi\) is well defined.
In the following we show that \(\Psi\) is a self-map. Set
\[
\Psi(v_1)=v.
\]
Then \(v-v_{\mathrm{ref}}\) has vanishing initial value. Hence, by \autoref{lem:deterministic-v},
\[
\|v-v_{\mathrm{ref}}\|_{\mathbb E_{1,\tau}}
\le
C(\omega)\bigl(
\|F_u(v_1+U)\|_{\mathbb E_{0,\tau}}
+
\|F_\Gamma(v_1+U)-\bigl(p(\varrho_0)-p_{\mathrm{ext}}\bigr)N\|_{\mathbb G_{\tau}}
\bigr),
\]
where \(C(\omega)>0\) is independent of \(\tau\).
Using \autoref{lem:pathwise-estimates-Fu-FGamma}, we obtain
\begin{equation*}
\begin{aligned}
\|v-v_{\mathrm{ref}}\|_{\mathbb E_{1,\tau}}
\le
C(\omega)\bigl(
\delta\,\bigl(1+R+M_{u_0}(\omega)+M_{\mathrm{sto}}(\omega)\bigr)
+
\tau^{1/p}
\bigr),
\end{aligned}
\end{equation*} 
where
\[
M_{\mathrm{sto}}(\omega)
=
\|U(\omega)\|_{\E_{1,\tau}^{\mathrm{sto}}} .
\]
Therefore, by choosing \(\delta>0\) and \(T>0\) sufficiently small, we obtain
\[
\|v-v_{\mathrm{ref}}\|_{\mathbb E_{1,\tau}}\le r.
\]
Since moreover \(v(0)=u_0=v_{\mathrm{ref}}(0)\), it follows that
\[
\Psi(v_1)\in \mathbb B_r(0,\tau).
\]
Thus \(\Psi\) is a self-map.
By similar means, we show that $\Psi$ is a contraction, using the estimates for differences given in \autoref{cor:pathwise-difference-stochastic-lagrangian} and \autoref{cor:pathwise-difference-nonlinearities}.
\end{step}  
\begin{step}[Localization and Fixed Point Iteration]\label{step picard proof}\mbox{} \\ 
We now localize the pathwise construction in order to obtain deterministic bounds for the constants appearing in the fixed-point estimates. 
For \(N\in\N\), define the \(\mathcal F_0\)-measurable sets
\[
\Omega_N
:=
\bigl\{ \omega \in \Omega \colon
\|\varrho_0(\omega)\|_{\rH^{1,q}(\Omega_0)}+\|\varrho_0(\omega)^{-1}\|_{\rH^{1,q}(\Omega_0)}+\|u_0(\omega)\|_{\rX_\gamma}\le N
\bigr\}.
\]
Then \((\Omega_N)_{N\in\N}\) is an increasing sequence of sets in \(\mathcal F_0\) with
\[
\Omega_N\subset \Omega_{N+1},
\ \text{ and }\ 
\bigcup_{N\in\N}\Omega_N=\Omega
\quad\text{a.s.}
\]
Next, we localize the time-dependent random quantities. Define \(\tau_N\) by
\[
\tau_N
:=
\inf\bigl\{
t\in[0,T_0 ]:
\|v_{\mathrm{ref}}\|_{\mathbb E_{1,t}}
+
\|U\|_{\E_{1,t}^{\mathrm{sto}} } + \Lambda(t) - \Lambda(0) + K_\alpha(t)
\ge N
\bigr\}\wedge T_0 .
\] 
Since the map
\[
t\longmapsto
\|v_{\mathrm{ref}}\|_{\mathbb E_{1,t}}
+
\|U\|_{\E_{1,t}^{\mathrm{sto}}}+\Lambda(t) - \Lambda(0) + K_\alpha(t) 
\] 
is adapted, continuous, and nondecreasing, \(\tau_N\) is an \((\mathcal F_t)_t\)-stopping time. Now fix \(N\in\N\).
Then on \(\Omega_N \times [0,\tau_N]\) the constants \(C(\omega)\) in
\autoref{lem:pathwise-estimates-Fu-FGamma} and
\autoref{cor:pathwise-difference-nonlinearities}
are bounded by a deterministic constant \(C_N>0\).
We may therefore choose numbers
\[
\delta_N\in(0,\delta_0]
\ \text{ and }\
T_N\in(0,T_0 ],
\] 
depending only on \(N\), such that the self-map and contraction estimates from \autoref{step 1 proof} hold with \(C(\omega)\) replaced by \(C_N\). Define
\[
\tilde{\tau}_N
:=
\tau_{\delta_N, R, \theta}
\wedge
\tau_N
\wedge T_N,
\] 
Then \(\tilde{\tau}_N\) is an \((\mathcal F_t)_t\)-stopping time, and on the localized set \(\Omega_N\) the fixed-point map
\[
\Psi:\mathbb B_r(0,\tilde{\tau}_N)\to \mathbb E_{1,\tilde{\tau}_N}
\]
is well defined and, by the same argument as in \autoref{step 1 proof}, is a strict contraction. Hence, for each \(N\in\N\), there exists a unique fixed point
\[
v_N(\omega)\in \mathbb B_r(0,\tilde{\tau}_N).
\]
For each \(N\in\N\), we now recover progressive measurability of the localized fixed point \(v_N\) by means of the stopped fixed point iteration on \([0,\tilde{\tau}_N]\). To this end, define recursively
\[
v_N^{(0)}:=v_{\mathrm{ref}},
\]
and
\[
v_N^{(m+1)}
:=
\mathbf 1_{\Omega_N}\,\Psi(v_N^{(m)})
+
\mathbf 1_{\Omega\setminus\Omega_N}\,v_{\mathrm{ref}},
\quad m\in\N_0.
\]
Since \(\Omega_N\in\mathcal F_0\subset\mathcal F_t\), it is enough to prove by induction that each iterate \(v_N^{(m)}\) is progressively measurable on \([0,\tilde{\tau}_N]\).
The initial iterate \(v_N^{(0)}=v_{\mathrm{ref}}\) is progressively measurable. Assume that \(v_N^{(m)}\) is progressively measurable. Then
\[
\bar u_N^{(m)}:=v_N^{(m)}+U
\]
is progressively measurable as well. Let
\(
\rX_N^{(m)},\, \mathrm Z_N^{(m)},\, J_N^{(m)}
\)
denote the stochastic Lagrangian quantities associated with \(\bar u_N^{(m)}\). Since \(\bar u_N^{(m)}\) is progressively measurable and the coefficients \(Q_k\) are smooth, the corresponding stochastic flow \(\rX_N^{(m)}\) is adapted and has continuous trajectories; hence \(\rX_N^{(m)}\) is progressively measurable. Therefore \(\nabla \rX_N^{(m)}\) is progressively measurable, and so are
\[
\mathrm Z_N^{(m)}=(\nabla \rX_N^{(m)})^{-1}
\ \text{ and }\
J_N^{(m)}=\det\nabla \rX_N^{(m)},
\]
by continuity of inversion and determinant on the neighborhood where \(\nabla \rX_N^{(m)}\) is invertible.
We then recover the transformed density by
\[
\bar\varrho_N^{(m)}:=\frac{\varrho_0}{J_N^{(m)}}.
\]
Thus \(\bar\varrho_N^{(m)}\) is progressively measurable.
On \(\Omega_N\), let  \(v_N^{(m+1)}\) be the unique solution of
\[
\left\{
\begin{aligned}
\partial_t v_N^{(m+1)}+\mathcal A(x,D)v_N^{(m+1)}
&=
F_u(\bar u_N^{(m)}),
\\
\mathcal B(x,D)v_N^{(m+1)}
&=
F_\Gamma(\bar u_N^{(m)}),
\\
v_N^{(m+1)}(0)&=u_0.
\end{aligned}
\right.
\]
Since \(\bar u_N^{(m)}\), \(\mathrm Z_N^{(m)}\), \(J_N^{(m)}\), and \(\bar\varrho_N^{(m)}\) are progressively measurable, the nonlinearities
\(
F_u(\bar u_N^{(m)}), \,
F_\Gamma(\bar u_N^{(m)})
\)
are progressively measurable as well. Moreover, the Lam\'e system defining \(v_N^{(m+1)}\) is deterministic, continuous, and causal. Consequently, \(v_N^{(m+1)}\) is progressively measurable. This proves the induction step.
Moreover, on \(\Omega_N\), the map \(\Psi\) is a strict contraction on \(\mathbb B_r(0,\tilde\tau_N)\). Therefore, for a.e. \(\omega\in\Omega_N\),
\[
v_N^{(m)}(\omega)
\to
v_N(\omega) \
\text{ in } \
\mathbb E_{1,\tilde\tau_N(\omega)}
\ \text{ as }m\to\infty,
\]
where
\(
v_N(\omega)\in \mathbb B_r(0,\tilde\tau_N(\omega))
\)
is the unique fixed point. Since the iterates \(v_N^{(m)}\) are progressively measurable, it follows that \(v_N\) is progressively measurable on \([0,\tilde\tau_N]\).
Finally, setting
\(
\bar u_N:=v_N+U,
\)
and defining the associated quantities
\(
\rX_N,\, \mathrm Z_N,\, J_N
\)
as in the pathwise construction, we obtain by the same argument that these processes are progressively measurable on \([0,\tilde\tau_N]\). We recover the transformed density by
\[
\bar\varrho_N:=\frac{\varrho_0}{J_N},
\]
which is progressively measurable as well. Since \(v_N\) is a fixed point of \(\Psi\), the tuple
\[
(\bar\varrho_N,v_N+ U ,\rX_N,\mathrm Z_N,J_N)
\] 
solves the transformed fixed-domain problem on \([0,\tilde\tau_N]\).
\end{step}

\begin{step}[Patching and reconstruction of the Eulerian solution]\mbox{} \\
For each \(N\in\N\), let
\(
\bar\varrho_N,\, v_N,\,
\rX_N,\,
\mathrm Z_N,\,
J_N
\)
denote the progressively measurable localized fixed point and the associated Lagrangian quantities on the stochastic interval \([0,\tilde\tau_N]\) obtained in \autoref{step picard proof}. To this end, define the disjoint sets
\[
A_1:=\Omega_1 \ \text{ and }\
A_N:=\Omega_N\setminus\Omega_{N-1},
\quad N\ge2.
\]
Then \((A_N)_{N\in\N}\) is a partition of \(\Omega\) up to a null set, and each \(A_N\in\mathcal F_0\).
Next, define the global stopping time
\[
\tau
:=
\sum_{N=1}^\infty \mathbf 1_{A_N}\,\tilde\tau_N
\]
and the patched Lagrangian solution by
\[
v
:=
\sum_{N=1}^\infty \mathbf 1_{A_N}\,v_N\ \text{ and }\
\bar\varrho
:=
\sum_{N=1}^\infty \mathbf 1_{A_N}\,\bar\varrho_N.
\]
Likewise, we define the quantities \(\rX\), \(\mathrm Z\), and \(J\).
Since \(A_N\in\mathcal F_0\subset\mathcal F_t\) and each \(\tilde\tau_N\) is an \((\mathcal F_t)_t\)-stopping time, it follows that \(\tau\) is an \((\mathcal F_t)_t\)-stopping time. Moreover, since each \(v_N\), \(\bar\varrho_N\), \(\rX_N\), \(\mathrm Z_N\), and \(J_N\) is progressively measurable, the patched objects \(v\), \(\bar\varrho\), \(\rX\), \(\mathrm Z\), and \(J\) are progressively measurable as well.
By construction, on each set \(A_N\) we have
\[
(v,\bar\varrho,\rX,\mathrm Z,J)
=
(v_N,\bar\varrho_N,\rX_N,\mathrm Z_N,J_N)
\ \text{ on }[0,\tau].
\]
Hence the tuple \((\bar\varrho,v+U ,\rX,\mathrm Z,J)\)  solves the transformed fixed-domain problem on the stochastic interval \([0,\tau]\), and
\[
\bar\varrho=\frac{\varrho_0}{J}
\ \text{ on }[0,\tau].
\]
{ 
After subtracting the common stochastic convolution $U$, the Lagrangian
remainder of any other solution with the same initial data, force and
Brownian motions belongs to the contraction ball on a sufficiently
short interval.
The contraction estimate and a standard continuity argument therefore
give uniqueness of the Lagrangian solution throughout the common
interval of existence.
}
Finally, we return to Eulerian coordinates by applying the inverse
Lagrangian transformation. Let
\[
\rF(t,\omega,\cdot):=\rX(t,\omega,\cdot)^{-1}
\]
denote the inverse flow map. We then define the moving domain and the
Eulerian unknowns by
\[
\Omega_t(\omega):=\rX(t,\omega,\Omega_0),
\]
and
\[
\varrho(t,\omega,x)
:=
\bar\varrho\bigl(t,\omega,\rF(t,\omega,x)\bigr)
\ \text{ and }\
u(t,\omega,x)
:=
\bigl(v+U\bigr)\bigl(t,\omega,\rF(t,\omega,x)\bigr)
\ \text{ for }\
x\in\Omega_t(\omega).
\]
Thus the original free-boundary solution in Eulerian variables is
recovered from the patched Lagrangian solution by composition with the
inverse transformation \(\rF\). In particular,
\[
(\varrho,u,(\Omega_t)_{t\in[0,\tau]})
\]
is a {pathwise unique} progressively measurable local strong
solution in the sense of \autoref{def: local strong solution} of the
stochastic free boundary value problem
\eqref{eq:stochastic-free-boundary-CNSE-RHS}.
\end{step}
\end{proof}

\appendix

\section{Auxiliary estimates for the stochastic Lagrangian flow}
\label{appendix}
\noindent
In this appendix, we collect the technical auxiliary results needed for the control of
the stochastic Lagrangian map in \autoref{prop:pathwise-stochastic-lagrangian}.
The first two lemmas are purely functional analytic, the third one gives the ODE
estimate for the auxiliary flow \(\mathrm Y\), and the fourth one provides the
time-dependent inverse and determinant bounds.

\begin{lem}
\label{lem:aux-space-lagrangian}
Let \(q\in(3,\infty)\). Then the following assertions hold.

\begin{enumerate}
\item[(i)] The spaces \(\rH^{1,q}(\Omega_0)\) and \(\rH^{2,q}(\Omega_0)\) are
Banach algebras. More precisely, there exists \(C>0\) depending on $q$ and $\Omega_0$ such that
\[
\|fg\|_{\rH^{1,q}(\Omega_0)}
\le
C\|f\|_{\rH^{1,q}(\Omega_0)}\|g\|_{\rH^{1,q}(\Omega_0)},
\]
and
\[
\|fg\|_{\rH^{2,q}(\Omega_0)}
\le
C\|f\|_{\rH^{2,q}(\Omega_0)}\|g\|_{\rH^{2,q}(\Omega_0)}.
\]

\item[(ii)] Let \(F\in \rC_b^2(\R^3;\R^m)\) and
\(Y\in \rH^{2,q}(\Omega_0;\R^3)\). Then
\[
\|F\circ Y\|_{\rH^{2,q}(\Omega_0)}
\le
C\|F\|_{\rC_b^2(\R^3)}
\bigl(
1+\|Y-\mathrm{id}\|_{\rH^{2,q}(\Omega_0)}
+\|Y-\mathrm{id}\|_{\rH^{2,q}(\Omega_0)}^2
\bigr).
\]

\item[(iii)] Let \(F\in \rC_b^3(\R^3;\R^m)\) and
\(Y_1,Y_2\in \rH^{2,q}(\Omega_0;\R^3)\). Then
\[
\|F\circ Y_1-F\circ Y_2\|_{\rH^{2,q}(\Omega_0)}
\le
C\|F\|_{\rC_b^3(\R^3)}
\bigl(
1+\|Y_1-\mathrm{id}\|_{\rH^{2,q}}
+\|Y_2-\mathrm{id}\|_{\rH^{2,q}}
\bigr)^2
\|Y_1-Y_2\|_{\rH^{2,q}(\Omega_0)}.
\]

\item[(iv)] There exist \(\eps_*>0\) and \(C>0\) depending on $q$ and $\Omega_0$ such that the
following holds. If
\[
A,B\in \rH^{1,q}(\Omega_0;\R^{3\times3}),
\qquad
\|A-I_3\|_{\rH^{1,q}(\Omega_0)}
+
\|B-I_3\|_{\rH^{1,q}(\Omega_0)}
\le \eps_*,
\]
then \(A(y)\) and \(B(y)\) are invertible for every \(y\in\Omega_0\), and
\[
\|A^{-1}-B^{-1}\|_{\rH^{1,q}(\Omega_0)}
+
\|\det A-\det B\|_{\rH^{1,q}(\Omega_0)}
\le
C\|A-B\|_{\rH^{1,q}(\Omega_0)}.
\]
In particular, taking \(B=I_3\),
\[
\|A^{-1}-I_3\|_{\rH^{1,q}(\Omega_0)}
+
\|\det A-1\|_{\rH^{1,q}(\Omega_0)}
\le
C\|A-I_3\|_{\rH^{1,q}(\Omega_0)}.
\]
\end{enumerate}
\end{lem}

\begin{proof}
(i). Since \(q>3\), we have the Sobolev embeddings
\[
\rH^{1,q}(\Omega_0)\hookrightarrow \rL^\infty(\Omega_0),
\qquad
\rH^{2,q}(\Omega_0)\hookrightarrow \rW^{1,\infty}(\Omega_0).
\]
Thus, (i) follows by differentiating the product.

(ii). Since \(F\in \rC_b^2(\R^3;\R^m)\), we have
\[
\|F\circ Y\|_{\rL^q(\Omega_0)}
\le
|\Omega_0|^{1/q}\|F\|_{\rL^\infty(\R^3)}
\le
C\|F\|_{\rC_b^2(\R^3)}.
\]
Moreover, by the chain rule,
\[
\nabla(F\circ Y)=DF(Y)\nabla Y,
\]
hence
\[
\|\nabla(F\circ Y)\|_{\rL^q(\Omega_0)}
\le
\|DF\|_{\rL^\infty(\R^3)}
\|\nabla Y\|_{\rL^q(\Omega_0)}
\le
C\|F\|_{\rC_b^2(\R^3)}\|Y\|_{\rH^{2,q}(\Omega_0)}.
\]
Differentiating once more, we obtain
\[
\nabla^2(F\circ Y)
=
D^2F(Y)[\nabla Y,\nabla Y]
+
DF(Y)\nabla^2Y.
\]
Therefore,
\[
\|\nabla^2(F\circ Y)\|_{\rL^q(\Omega_0)}
\le
\|D^2F\|_{\rL^\infty(\R^3)}
\|\nabla Y\|_{\rL^\infty(\Omega_0)}
\|\nabla Y\|_{\rL^q(\Omega_0)}
+
\|DF\|_{\rL^\infty(\R^3)}
\|\nabla^2Y\|_{\rL^q(\Omega_0)}.
\]
Since \(q>3\), we have the Sobolev embedding
\[
\rH^{2,q}(\Omega_0)\hookrightarrow \rW^{1,\infty}(\Omega_0),
\]
so
\[
\|\nabla Y\|_{\rL^\infty(\Omega_0)}
\le
C\|Y\|_{\rH^{2,q}(\Omega_0)}.
\]
Combining the previous estimates, we get
\[
\|F\circ Y\|_{\rH^{2,q}(\Omega_0)}
\le
C\|F\|_{\rC_b^2(\R^3)}
\Bigl(
1+\|Y\|_{\rH^{2,q}(\Omega_0)}
+\|Y\|_{\rH^{2,q}(\Omega_0)}^2
\Bigr).
\]
Finally, since \(\mathrm{id}\in \rH^{2,q}(\Omega_0;\R^3)\),
\[
\|Y\|_{\rH^{2,q}(\Omega_0)}
\le
\|Y-\mathrm{id}\|_{\rH^{2,q}(\Omega_0)}
+
\|\mathrm{id}\|_{\rH^{2,q}(\Omega_0)}
\le
C\Bigl(1+\|Y-\mathrm{id}\|_{\rH^{2,q}(\Omega_0)}\Bigr),
\]
and the claim follows.

(iii). Define $
Z_\tau:=Y_2+\tau(Y_1-Y_2),$ for $\tau\in[0,1].$ Then
\[
F(Y_1)-F(Y_2)
=
\int_0^1 DF(Z_\tau)\,(Y_1-Y_2)\,d\tau.
\]
By (i) and (ii), applied to \(DF\in \rC_b^2\), we get
\[
\|DF(Z_\tau)(Y_1-Y_2)\|_{\rH^{2,q}}
\le
C\|F\|_{\rC_b^3}
\Bigl(
1+\|Z_\tau-\mathrm{id}\|_{\rH^{2,q}}
+\|Z_\tau-\mathrm{id}\|_{\rH^{2,q}}^2
\Bigr)
\|Y_1-Y_2\|_{\rH^{2,q}}.
\]
Since
\[
\|Z_\tau-\mathrm{id}\|_{\rH^{2,q}}
\le
\|Y_1-\mathrm{id}\|_{\rH^{2,q}}
+
\|Y_2-\mathrm{id}\|_{\rH^{2,q}},
\]
integration in \(\tau\) gives (iii).

(iv). Since \(q>3\), we have the Sobolev embedding $\rH^{1,q}(\Omega_0)\hookrightarrow \rL^\infty(\Omega_0),$ and by part (i), \(\rH^{1,q}(\Omega_0)\) is a Banach algebra. Thus there exists
\(C_{\mathrm{alg}}>0\) such that
\[
\|FG\|_{\rH^{1,q}(\Omega_0)}
\le
C_{\mathrm{alg}}
\|F\|_{\rH^{1,q}(\Omega_0)}
\|G\|_{\rH^{1,q}(\Omega_0)}.
\]
Choose \(\eps_*>0\) so small that $ C_{\mathrm{alg}}\eps_*<\frac12$
and, by the Sobolev embedding,
\[
\|M-I_3\|_{\rH^{1,q}(\Omega_0)}\le \eps_*
\quad\Longrightarrow\quad
\|M-I_3\|_{\rL^\infty(\Omega_0)}\le \frac12.
\]
Hence every such matrix field \(M\) is pointwise invertible. Moreover, setting
\(N:=I_3-M\), we have $
C_{\mathrm{alg}}\|N\|_{\rH^{1,q}(\Omega_0)}<1,$ so the Neumann series $
\sum_{k=0}^\infty N^k$ converges in the Banach algebra
\(\rH^{1,q}(\Omega_0;\R^{3\times3})\). Its sum is \(M^{-1}\). In particular $
\|M^{-1}\|_{\rH^{1,q}(\Omega_0)}\le C,$ where \(C\) depends only on \(q\) and \(\Omega_0\).

Applying this to \(A\) and \(B\), we get
\[
A^{-1}-B^{-1}=A^{-1}(B-A)B^{-1},
\]
and therefore, by the Banach algebra property,
\[
\|A^{-1}-B^{-1}\|_{\rH^{1,q}(\Omega_0)}
\le
C\|A-B\|_{\rH^{1,q}(\Omega_0)}.
\]

For the determinant, \(\det A-\det B\) is a finite sum of terms, each
containing exactly one factor of \(A-B\), while all remaining factors are
entries of \(A\) or \(B\). Since \(\|A\|_{\rH^{1,q}}\) and \(\|B\|_{\rH^{1,q}}\)
are uniformly bounded on
\[
\|A-I_3\|_{\rH^{1,q}}+\|B-I_3\|_{\rH^{1,q}}\le \eps_*,
\]
the Banach algebra property yields
\[
\|\det A-\det B\|_{\rH^{1,q}(\Omega_0)}
\le
C\|A-B\|_{\rH^{1,q}(\Omega_0)}.
\]
This concludes the proof.
\end{proof}

We now turn to the auxiliary flow \(\mathrm Y\) and establish the basic well-posedness and regularity estimates for its pathwise ODE.

\begin{lem}
\label{lem:aux-flow-lagrangian}
Let $T>0$, \(p\in(1,\infty)\), \(q\in(3,\infty)\), 
\(
a\in \rC([0,T];\rC_b^3(\R^3;\R^{3\times3})),
\)
\(
b\in \rL^p(0,T;\rH^{2,q}(\Omega_0;\R^3)).
\)
Consider the integral equation
\begin{equation}
\label{eq:aux-flow-ode}
Y_t
=
\mathrm{id}
+
\int_0^t a_s(Y_s)\,b(s)\,ds.
\end{equation}
Then there exists a unique solution
\[
Y-\mathrm{id}
\in
\rC([0,T];\rH^{2,q}(\Omega_0;\R^3)),
\qquad
 \frac{\d }{\d t} Y
\in
\rL^p(0,T;\rH^{2,q}(\Omega_0;\R^3)).
\]
Moreover, for
\(
L_a(t):=1+\sup_{0\le s\le t}\|a_s\|_{\rC_b^3(\R^3)},
\text{ and }
B_b(t):=\|b\|_{\rL^p(0,t;\rH^{2,q}(\Omega_0))},
\)
there exists \(C>0\) depending only on $p$, $q$, $\Omega_0$ such that
\begin{equation}
\label{eq:aux-flow-bootstrap}
\|Y-\mathrm{id}\|_{\rL^\infty(0,t;\rH^{2,q})}
\le
C\,L_a(t)\,t^{1-\frac1p}B_b(t)\,
\Bigl(
1+\|Y-\mathrm{id}\|_{\rL^\infty(0,t;\rH^{2,q})}
+\|Y-\mathrm{id}\|_{\rL^\infty(0,t;\rH^{2,q})}^2
\Bigr)
\end{equation}
for every \(t\in[0,T]\). In particular, if $
\beta_{a,b}(t):=C\,L_a(t)\,t^{1-\frac1p}B_b(t)\le \frac18,$ then
\begin{equation}
\label{eq:aux-flow-smalltime}
\|Y-\mathrm{id}\|_{\rL^\infty(0,t;\rH^{2,q})}
\le 2,
\qquad
\|Y-\mathrm{id}\|_{\rL^\infty(0,t;\rH^{2,q})}
\le 7\,\beta_{a,b}(t),
\end{equation}
and
\begin{equation}
\label{eq:aux-flow-derivative}
\|\frac{\d}{\d t} Y\|_{\rL^p(0,t;\rH^{2,q})}
\le
C\,L_a(t)\,\bigl(1+\beta_{a,b}(t)\bigr)\,B_b(t).
\end{equation}
Moreover, for every \(\theta\in(0,1)\) and \(0<t\le T\wedge1\)
\begin{equation}
\label{eq:aux-flow-fractional}
\|Y-\mathrm{id}\|_{\rH^{\theta,p}(0,t;\rH^{2,q})}
\le
C\,t^{1-\theta}\,
\|\frac{\d}{\d t}   Y\|_{\rL^p(0,t;\rH^{2,q})}.
\end{equation}
\end{lem}

\begin{proof}
Set \(E_\tau:=\rC([0,\tau];\rH^{2,q}(\Omega_0;\R^3))\), and for \(Z\in E_\tau\) define
\[
(\Phi Z)(t):=\mathrm{id}+\int_0^t a_s(Z_s)\,b(s)\,ds.
\]
By \autoref{lem:aux-space-lagrangian}(ii), for every \(Z\in E_\tau\),
\[
\|a_s(Z_s)\|_{\rH^{2,q}}
\le
C\,L_a(\tau)\Bigl(1+\|Z_s-\mathrm{id}\|_{\rH^{2,q}}+\|Z_s-\mathrm{id}\|_{\rH^{2,q}}^2\Bigr).
\]
Using also the Banach algebra property from \autoref{lem:aux-space-lagrangian}(i), we obtain on the ball
\[
\mathbb B_R^\tau
:=
\Bigl\{
Z\in E_\tau:\ \|Z-\mathrm{id}\|_{\rL^\infty(0,\tau;\rH^{2,q})}\le R
\Bigr\}
\]
the estimate
\[
\|\Phi Z-\mathrm{id}\|_{\rL^\infty(0,\tau;\rH^{2,q})}
\le
C\,L_a(\tau)\,\tau^{1-\frac1p}B_b(\tau)\,(1+R+R^2).
\]
Likewise, by \autoref{lem:aux-space-lagrangian}(iii) and again (i),
\[
\|\Phi Z_1-\Phi Z_2\|_{\rL^\infty(0,\tau;\rH^{2,q})}
\le
C\,L_a(\tau)\,\tau^{1-\frac1p}B_b(\tau)\,(1+R)^2
\|Z_1-Z_2\|_{\rL^\infty(0,\tau;\rH^{2,q})}.
\]
Hence, for \(\tau>0\) small enough, \(\Phi\) is a contraction on \(\mathbb B_R^\tau\), so
\eqref{eq:aux-flow-ode} has a unique solution on \([0,\tau]\).

Let \([0,T_{\max})\) be the maximal interval of existence of the local solution.
We claim that
\[
\sup_{0\le s<T_{\max}}\|Y_s-\mathrm{id}\|_{\rH^{2,q}(\Omega_0)}<\infty.
\]

We first estimate the first derivatives. Differentiating \eqref{eq:aux-flow-ode} once in space, we obtain
\[
\nabla Y_t
=
I_3
+
\int_0^t
\Bigl(
D a_s(Y_s)\,\nabla Y_s\, b(s)
+
a_s(Y_s)\,\nabla b(s)
\Bigr)\,ds,
\]
where all products/contractions are understood componentwise. Hence, using
\autoref{lem:aux-space-lagrangian}(i)-(ii) together with
\(\rH^{2,q}\hookrightarrow \rW^{1,\infty}\), which holds since $q>3$, we have
\[
\|\nabla Y_t\|_{\rL^\infty}
\le
1
+
C\,L_a(T)\int_0^t
\|b(s)\|_{\rH^{2,q}}
\Bigl(
1+\|\nabla Y_s\|_{\rL^\infty}
\Bigr)\,ds.
\]
Setting $
G(t):=\|\nabla Y_t\|_{\rL^\infty},$ Gronwall's lemma gives
\[
\sup_{0\le s<T_{\max}}G(s)<\infty.
\]

We next estimate the second derivatives. Differentiating \eqref{eq:aux-flow-ode} twice in space and using the chain and product rules, we obtain schematically
\[
\nabla^2\!\bigl(a_s(Y_s)b(s)\bigr)
=
\Bigl(
D^2 a_s(Y_s)[\nabla Y_s,\nabla Y_s]
+
D a_s(Y_s)\,\nabla^2 Y_s
\Bigr)b(s)
+
2\bigl(D a_s(Y_s)\,\nabla Y_s\bigr)\nabla b(s)
+
a_s(Y_s)\,\nabla^2 b(s).
\]
Therefore, by \autoref{lem:aux-space-lagrangian}(i)-(ii), the boundedness of the derivatives of \(a\), and the embedding
\(\rH^{2,q}\hookrightarrow \rW^{1,\infty}\),
\[
\|\nabla^2 Y_t\|_{\rL^q}
\le
C\,L_a(T)\int_0^t
\|b(s)\|_{\rH^{2,q}}
\Bigl(
1+\|\nabla Y_s\|_{\rL^\infty}^2+\|\nabla^2 Y_s\|_{\rL^q}
\Bigr)\,ds.
\]
Since \(\sup_{0\le s<T_{\max}}\|\nabla Y_s\|_{\rL^\infty}<\infty\) and
\(b\in \rL^1(0,T;\rH^{2,q})\), another application of Gronwall's lemma gives
\[
\sup_{0\le s<T_{\max}}\|\nabla^2 Y_s\|_{\rL^q}<\infty.
\]

Finally, from \eqref{eq:aux-flow-ode} and the boundedness of \(a\),
\[
\|Y_t-\mathrm{id}\|_{\rL^q}
\le
C\,L_a(T)\int_0^t \|b(s)\|_{\rL^q}\,ds
\le
C\,L_a(T)\int_0^t \|b(s)\|_{\rH^{2,q}}\,ds.
\]
Since \(\Omega_0\) is bounded,
\[
\|\nabla(Y_t-\mathrm{id})\|_{L^q}
=
\|\nabla Y_t-I_3\|_{L^q}
\le
|\Omega_0|^{1/q}\|\nabla Y_t-I_3\|_{L^\infty}
\lesssim 1+\|\nabla Y_t\|_{L^\infty}.
\]
Combining this with the \(L^q\)-bound for \(Y_t-\mathrm{id}\) and the
\(L^q\)-bound for \(\nabla^2Y_t\), we obtain
\[
\sup_{0\le s<T_{\max}}\|Y_s-\mathrm{id}\|_{H^{2,q}(\Omega_0)}<\infty.
\]
Hence the local fixed-point argument can be restarted at any time
\(t_0<T_{\max}\), using  \(L_a(T)\){ ,}  an upper bound for \(\|Y_{t_0}-\mathrm{id}\|_{\rH^{2,q}}\), and uniform absolute continuity of \(\int_I\|b(s)\|_{\rH^{2,q}}\,\d s\) over intervals \(I\subset[0,T]\), 
which is uniform by the previous estimate. This excludes finite-time blow-up,
and therefore \(T_{\max}=T\).

 Moreover,
\begin{equation}
    Y'(t)=a_t(Y_t)b(t)\qquad\text{for a.e. }t,
    \label{eq: aux 1}
\end{equation}
so \(\partial_tY\in \rL^p(0,T;\rH^{2,q})\).

Now set $
Y_*(t):=\|Y-\mathrm{id}\|_{\rL^\infty(0,t;\rH^{2,q})}.$ From \eqref{eq:aux-flow-ode}, H\"older's inequality, and \autoref{lem:aux-space-lagrangian}(i)-(ii),
\[
Y_*(t)
\le
C\,L_a(t)\,t^{1-\frac1p}B_b(t)\Bigl(1+Y_*(t)+Y_*(t)^2\Bigr),
\]
which is \eqref{eq:aux-flow-bootstrap}. If
\[
\beta _{a,b}(t):=C\,L_a(t)\,t^{1-\frac1p}B_b(t)\le \frac18,
\]
then, since \(Y_*(0)=0\) and \(Y_*\) is continuous, reasoning by contradiction it is possible to check that $
Y_*(t)\le 2.$ Substituting this into \eqref{eq:aux-flow-bootstrap}, we get $
Y_*(t)\le 7\,\beta_{a,b}(t),$ which concludes \eqref{eq:aux-flow-smalltime}.

Next, applying \autoref{lem:aux-space-lagrangian}(i)-(ii) to \eqref{eq: aux 1}, we obtain
\[
\|\frac{\d}{\d t}Y\|_{\rL^p(0,t;\rH^{2,q})}
\le
C\,L_a(t)\Bigl(1+Y_*(t)+Y_*(t)^2\Bigr)B_b(t).
\]
If \(\beta_{a,b}(t)\le \frac18\), then \(Y_*(t)\le 7\beta_{a,b}(t)\), and therefore
\[
\|\frac{\d}{\d t}Y\|_{\rL^p(0,t;\rH^{2,q})}
\le
C\,L_a(t)\,\bigl(1+\beta _{a,b}(t)\bigr)\,B_b(t),
\]
which is \eqref{eq:aux-flow-derivative}.

Finally, since \(Y(0)=\mathrm{id}\) by \eqref{eq:aux-flow-ode},
{ the standard interpolation estimate for functions with zero
initial value} yields
\[
\|Y-\mathrm{id}\|_{\rH^{\theta,p}(0,t;\rH^{2,q})}
\le
C\,t^{1-\theta}\,
\|\frac{\d}{\d t}Y\|_{\rL^p(0,t;\rH^{2,q})},
\]
that is, \eqref{eq:aux-flow-fractional}.
\end{proof}
\medskip

Finally, we state the time-dependent inverse and determinant estimates needed to control \(\mathrm Z\) and \(J\) once \(\nabla \rX\) stays close to \(I_3\).
\begin{lem}
\label{lem:aux-time-inverse-det}
Let \(q\in(3,\infty)\), \(p\in(1,\infty)\),
\(0<\theta<\eta<1\), \(\eta>1/p\), and \(0<t\le1\). Let
\[
A-I_3
\in
\rL^\infty(0,t;\rH^{1,q}(\Omega_0;\R^{3\times 3}))
\cap
\rW^{\eta,p}(0,t;\rH^{1,q}(\Omega_0;\R^{3\times 3})).
\]
Assume that \(A(0)=I_3\) and
\[
\|A-I_3\|_{\rL^\infty(0,t;\rH^{1,q}(\Omega_0))}\le \eps_*,
\]
where \(\eps_*\) is the constant from
\autoref{lem:aux-space-lagrangian}(iv). Define
\(Z:=A^{-1}\) and \(J:=\det A\). Then
\[
Z-I_3,\quad J-1
\in
\rL^\infty(0,t;\rH^{1,q}(\Omega_0))
\cap
\rH^{\theta,p}(0,t;\rH^{1,q}(\Omega_0)),
\]
and there exists \(C=C(p,q,\theta,\eta,\Omega_0)>0\) such that
\[
\|Z-I_3\|_{\rL^\infty(0,t;\rH^{1,q})}
+
\|J-1\|_{\rL^\infty(0,t;\rH^{1,q})}
\le
C\|A-I_3\|_{\rL^\infty(0,t;\rH^{1,q})},
\]
and
\[
\|Z-I_3\|_{\rH^{\theta,p}(0,t;\rH^{1,q})}
+
\|J-1\|_{\rH^{\theta,p}(0,t;\rH^{1,q})}
\le
C\Bigl(
t^{\eta-\theta}\|A-I_3\|_{\rW^{\eta,p}(0,t;\rH^{1,q})}
+
t^{1/p}\|A-I_3\|_{\rL^\infty(0,t;\rH^{1,q})}
\Bigr).
\]
The same estimates hold for \(\Phi(A)-\Phi(I_3)\) when \(\Phi\) is smooth
between the corresponding \(\rH^{1,q}\) spaces on this neighbourhood,
with bounded first derivative. The constant also depends on \(\Phi\).
\end{lem}

\begin{proof}
The \(\rL^\infty_t\rH^{1,q}_x\)-bound follows by applying
\autoref{lem:aux-space-lagrangian}(iv) at each fixed time and taking the
supremum in time. In particular,
\[
\sup_{\tau\in[0,t]}
\Bigl(
\|A(\tau)^{-1}\|_{\rH^{1,q}(\Omega_0)}
+
\|\det A(\tau)\|_{\rH^{1,q}(\Omega_0)}
\Bigr)
\le C.
\]

For the fractional seminorm, for \(r,s\in[0,t]\) we have
\[
Z(r)-Z(s)
=
A(r)^{-1}\bigl(A(s)-A(r)\bigr)A(s)^{-1}.
\]
Hence, using the Banach algebra property and the uniform bound above,
\[
\|Z(r)-Z(s)\|_{\rH^{1,q}}
\le
C\|A(r)-A(s)\|_{\rH^{1,q}}.
\]
Likewise, since \(\det A\) is a polynomial of degree \(3\) in the entries of
\(A\), the difference \(\det A(r)-\det A(s)\) is a finite sum of terms each
containing exactly one factor of \(A(r)-A(s)\), while all remaining factors are
entries of \(A(r)\) or \(A(s)\). Since \(\rH^{1,q}(\Omega_0)\) is a Banach
algebra and $
\sup_{\tau\in[0,t]}\|A(\tau)\|_{\rH^{1,q}(\Omega_0)}\le C,$ we obtain
\[
\|J(r)-J(s)\|_{\rH^{1,q}}
=
\|\det A(r)-\det A(s)\|_{\rH^{1,q}}
\le
C\|A(r)-A(s)\|_{\rH^{1,q}}.
\]
After raising to the power \(p\), dividing by \(|r-s|^{1+\eta p}\), and
integrating over \((0,t)^2\), we obtain
\[
[Z-I_3]_{\rW^{\eta,p}(0,t;\rH^{1,q})}
+
[J-1]_{\rW^{\eta,p}(0,t;\rH^{1,q})}
\le
C[A-I_3]_{\rW^{\eta,p}(0,t;\rH^{1,q})}.
\]

Moreover, by the pointwise-in-time estimate from
\autoref{lem:aux-space-lagrangian}(iv) with \(B=I_3\),
\[
\|Z(\tau)-I_3\|_{\rH^{1,q}}
+
\|J(\tau)-1\|_{\rH^{1,q}}
\le
C\|A(\tau)-I_3\|_{\rH^{1,q}}
\qquad\text{for all }\tau\in[0,t].
\]
Hence
\[
\|Z-I_3\|_{\rL^p(0,t;\rH^{1,q})}
+
\|J-1\|_{\rL^p(0,t;\rH^{1,q})}
\le
Ct^{1/p}\|A-I_3\|_{\rL^\infty(0,t;\rH^{1,q})}.
\]

{ 
For \(E=\rH^{1,q}(\Omega_0)\) or \(E=\rH^{2,q}(\Omega_0)\),
the standard embedding \(\rW^{\eta,p}(\R;E)\hookrightarrow
\rH^{\theta,p}(\R;E)\), together with zero-trace extension and rescaling, gives
\begin{equation}\label{eq:aux-time-embedding}
\|h\|_{\rH^{\theta,p}(0,t;E)}
\le Ct^{\eta-\theta}\|h\|_{\rW^{\eta,p}(0,t;E)},
\qquad h(0)=0,
\end{equation}
with \(C\) independent of \(t\), since \(\eta>\max\{\theta,1/p\}\)
and \(0<t\le1\).
Applying this to \(Z-I_3\) and \(J-1\), together with the preceding estimates
and \(t^{\eta-\theta}\le1\), proves the asserted
\(\rH^{\theta,p}\)-bound.
For \(\Phi(A)-\Phi(I_3)\), its bounded derivative gives the same pointwise
and increment bounds, so the same argument applies.
}
This concludes the proof.
\end{proof}
\noindent 
{\bf Acknowledgements. }{\small All three authors acknowledge the support from the DFG project FOR~5528.}
\nocite{*}
\bibliographystyle{plain}
\bibliography{biblio}

\end{document}